\documentclass[english]{article}

\usepackage{amsthm,amsmath,amssymb}
\usepackage[authoryear]{natbib}
\usepackage[colorlinks=true,linkcolor=black,citecolor=black,urlcolor=black,breaklinks]{hyperref}
\usepackage{babel}
\usepackage{caption, subcaption}
\usepackage{graphicx} % support the \includegraphics command and options
\usepackage{algorithm}
\usepackage{custom_tex}
\usepackage{grffile}
\usepackage{comment}
\usepackage{geometry} % to change the page dimensions
\newcommand{\ep}{\varepsilon}

\newcommand{\N}{\mathcal{N}}

\newcommand{\E}{\mathbb{E}}

\newcommand{\A}{\mathcal{A}}

\newcommand{\supp}{\textnormal{supp}}
\newcommand{\bitem}{\begin{itemize}}
\newcommand{\eitem}{\end{itemize}}
\newcommand{\benum}{\begin{enumerate}}
\newcommand{\eenum}{\end{enumerate}}
\newcommand{\beq}{\begin{equation}}
\newcommand{\eeq}{\end{equation}}
\newcommand{\beqs}{\begin{equation*}}
\newcommand{\eeqs}{\end{equation*}}

\newtheorem{theorem}{Theorem}[section]

\newtheorem{lemma}[theorem]{Lemma}

\newtheorem{corollary}[theorem]{Corollary}

\newtheorem{thm}{Theorem}[section]

\newtheorem{prop}[thm]{Proposition}

\theoremstyle{thm}

\begin{document}
\title{Permutation methods for factor analysis and PCA}

\date{\today}
\author{Edgar Dobriban\footnote{Department of Statistics, The Wharton School, University of Pennsylvania. E-mail: \texttt{dobriban@wharton.upenn.edu}. 
We thank Andreas Buja, David Donoho, Alexei Onatski, and Art Owen for stimulating discussions. We are grateful to Jingshu Wang for feedback on the manuscript. The first version of the manuscript had the title "Factor selection by permutation".}}
\maketitle
\abstract{
Researchers often have datasets measuring features $x_{ij}$ of samples, such as test scores of students. In factor analysis and PCA, these features are thought to be influenced by unobserved factors, such as skills. Can we determine how many components affect the data? This is an important problem, because it has a large impact on all downstream data analysis. Consequently, many approaches have been developed to address it. 
\emph{Parallel Analysis} is a popular permutation method. It works by randomly scrambling each feature of the data. It 
selects components if their singular values are larger than those of the permuted data. 
Despite widespread use in leading textbooks and scientific publications, as well as empirical evidence for its accuracy, it currently has no theoretical justification. %This prevents us from knowing when it will work in the future.

In this paper, we show that the parallel analysis permutation method consistently selects the large components in certain high-dimensional factor models. However, it does not select the smaller components.
The intuition is that permutations keep the noise invariant, while ``destroying'' the low-rank signal. This provides justification for permutation methods in PCA and factor models under some conditions. Our work uncovers drawbacks of permutation methods, and paves the way to improvements. 
}

\section{Introduction}

\subsection{Factor Analysis and PCA}
Factor Analysis (FA) and Principal Component Analysis (PCA), the unsupervised discovery of components governing variation in the data, is performed routinely by scientists and social scientists in thousands of studies every year. In FA and PCA, we measure a number $p$ of indicators (features, covariates) for a set of $n$ samples. In \cite{spearman1904general}'s original application, this involved $p$ test scores of $n$ students. In another important application to finance, we measure the return for $n$ assets over $p$ (or $T$) time periods. The goal is to identify the common factors driving variation in the data, such as skills in Spearman's example, or systemic risks in finance. The setup for FA and PCA is similar, while not exactly the same \citep[see e.g.,][]{jolliffe2002principal}, and hence we will focus on factor analysis for clarity in most of the paper. 

Since Spearman's time factor analysis has found a wide range of applications in a variety of fields, becoming one of the most widely used statistical methods. Applications abound in psychology and education \citep{fabrigar1999evaluating, costello2005best, brown2014confirmatory}, public health \citep{goetz2008movement}, management/ marketing \citep{churchill1979paradigm, stewart1981application, parasuraman1988servqual}, economics/ finance \citep{bai2008large}, and genomics \citep{ leek2008general, lin2016simultaneous}.

\begin{figure}
  \centering
  \includegraphics[scale=0.5]{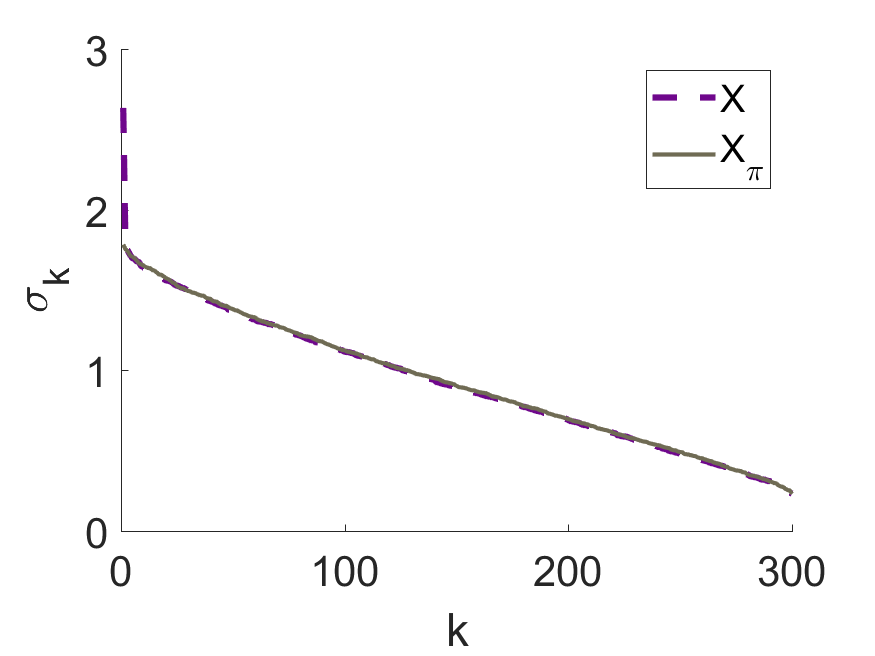}
\caption{Visual representation of the permutation method Parallel Analysis (PA). We have an $n \times p$ data matrix $X$ measuring $p$ features of $n$ datapoints. We want to determine how many unobserved factors or components affect the data. We examine the \emph{scree plot} of the singular values of $X$, i.e., the plot of singular values in a decreasing order. Classical methods such as Cattel's scree plot look for the "elbow" in this plot. Instead of such a subjective rule that may be biased by the judgement of the user, we consider a more objective permutation method. We permute the entries within each column of $X$ independently, possibly several times, to get "null" or "fake" data matrices $X_\pi$. We plot some fixed percentile of the largest, second largest etc., singular values of these matrices. We select the components of $X$ whose singular values are larger than the permuted ones. Here, only one factor is selected.}
\label{fig:PA_vis}
\end{figure}

The most widely used approach to FA relies on the the common-factor model \citep[e.g.,][etc]{thurstone1947multiple,anderson1958introduction}. For each sample $i$, the $j$-th indicator $x_{ij}$ is a linear function of one or more common factors $\eta_{ik}$ and one unique factor (or idiosyncratic noise) $\ep_{ij}$:
\beq
\label{feq}
x_{ij} = \sum_{k=1}^r \eta_{ik} \lambda_{jk} + \ep_{ij}.
\eeq
The factor values $\eta_{ik}$ and the factor loadings $\lambda_{jk}$ are not observed. In Spearman's example, $x_{ij}$ is the test score of student $i$ on test $j$, the $r$  factors are interpreted as skills, $\eta_{ik}$ is student $i$'s level on the $k$-th skill, and $\lambda_{jk}$ is the relevance of the $k$-th skill to the $j$-th test. %The goal is to determine the number of skills. 

FA is merely one step beyond linear regression. In linear regression, $\eta_{ik}$ are observed, while in FA they are not. This simplicity is deceiving, however, and FA can be surprisingly difficult. A widely cited tutorial on FA notes:  
``\emph{Perhaps more than any other commonly used statistical
method, FA requires a researcher to make a
number of important decisions with respect to how the
analysis is performed}'' \citep{fabrigar1999evaluating}.

One of the key problems in FA is to select the number of factors. For instance, how many skills control test scores? This is well known to have a large impact on the later steps of data analysis \cite[e.g.,][]{hayton2004factor,brown2014confirmatory}. The standard textbook \cite{brown2014confirmatory} calls it ``\emph{the most crucial decision}'' in exploratory FA. 

The factor selection problem is also important in principal component analysis (PCA). While PCA and factor analysis are not the same \citep[see e.g.,][for a clear explanation]{jolliffe2002principal}, in practice permutation methods are very popular to select the number of PCs \citep[e.g.,][]{friedman2009elements,zhou2017eigenvalue}. Our work also bears on PCA. 

Factor models are also well studied in econometrics, \citep[e.g.,][etc]{bai2008large,onatski2009testing,onatski2010determining,fan2011highd,bai2012statistical}. In that setting, the factors are assumed to be so strong that identifying the significant factors is trivial. In contrast, we study settings with weaker, ``emergent'' factors. These are common in the behavioral and biological sciences.

\begin{figure}
\centering
\begin{minipage}{.5\textwidth}
  \centering
  \includegraphics[scale=0.5]{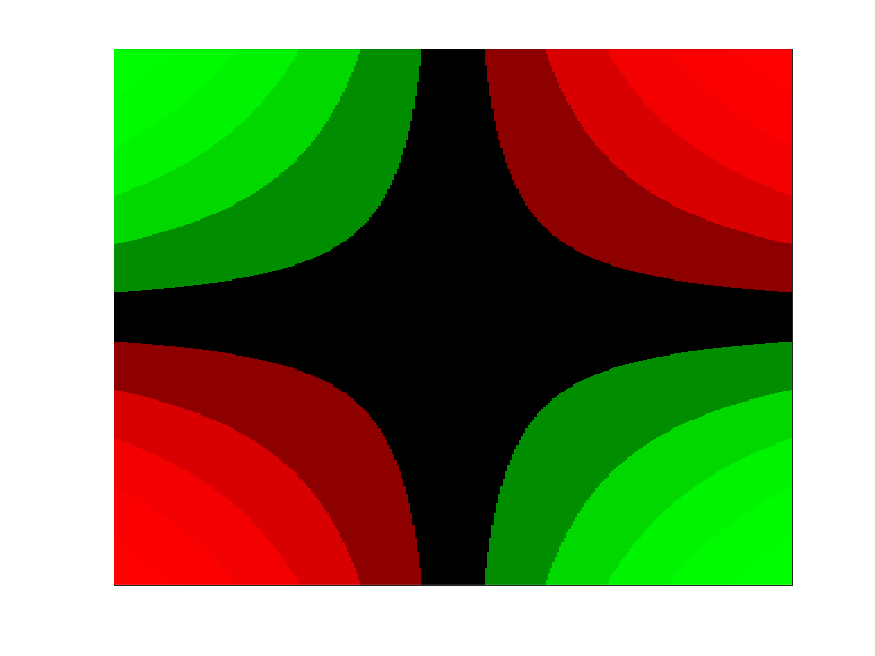}

\end{minipage}%
\begin{minipage}{.5\textwidth}
  \centering
  \includegraphics[scale=0.5]{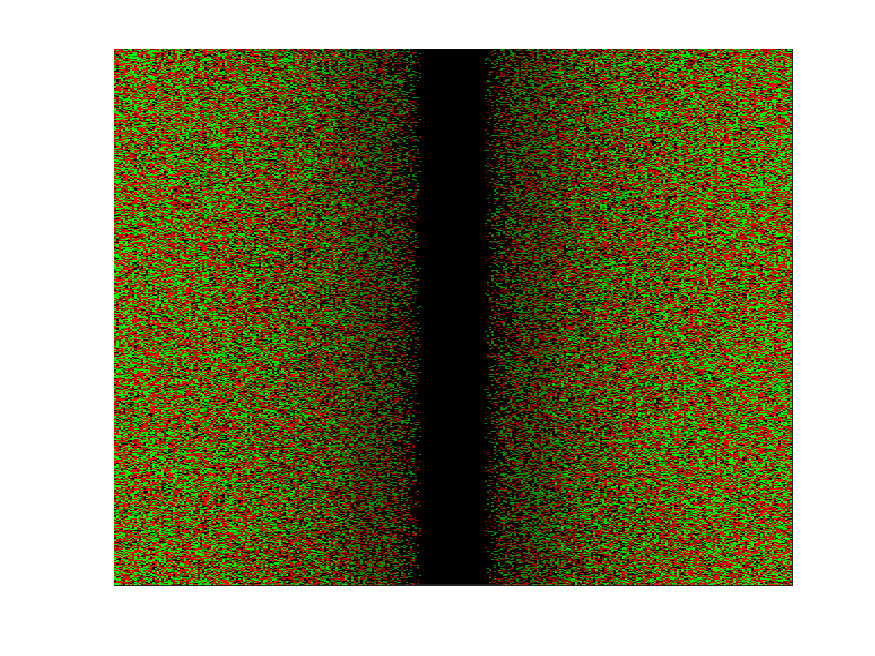}
\end{minipage}
\caption{How does PA work? Given a ``smooth'' signal $S$ of rank one (left), a random permutation transforms it into a ``rough'', noise-like matrix $S_\pi$. The permuted matrix is typically of full rank, and its operator norm is much smaller than that of the signal matrix. Thus, $S_\pi$ does not perturb the permuted noise matrix $N_\pi$ significantly, which allows the estimation of the noise level $\|N_\pi\|_{op} =_d \|N\|_{op}$ (equality in distribution). Then, factors above the noise level are selected.}
\label{fig:sig_des}
\end{figure}

\subsection{Parallel Analysis}

 Parallel Analysis (PA) \citep{horn:1965, buja:eyob:1992} is one of the most popular methods for selecting the number of factors.  In the widely used version proposed by \cite{buja:eyob:1992}, we start with the $n \times p$ data matrix $X$ containing the measurements $x_{ij}$, $i=1,\ldots,n$, $j=1,\ldots,p$. We generate a matrix $X_\pi$ by permuting the entries in each column of $X$ separately.   Here $\pi = (\pi_1,\pi_2,\ldots,\pi_p)$ is a permutation array, which is a collection of permutations $\pi_j$ of $\{1, \ldots, n\}$. The permutation $\pi_j$ permutes the $j$-th column of $X$, so $X_\pi$ has entries $(X_\pi)_{ij} = X_{\pi_j(i),j}$. We repeat this procedure a few times.

We select the first factor if the top singular value of $X$ is larger than a fixed percentile of the top singular values of the permuted matrices. One can use the median, 95\%-th, or 100\%-th percentile. If the first factor is selected, then we repeat the same procedure for the second largest singular value of $X$, comparing with the second singular values of permuted matrices, and so on. We stop when a factor is not selected. %This is a sequential testing procedure. %As a variant, one can  

Figure \ref{fig:PA_vis} illustrates parallel analysis. A data matrix $X$ is randomly permuted, and the singular values of both $X$ and $X_\pi$ are plotted in decreasing order. Only the first singular value of $X$ is larger than that of $X_\pi$, so one factor is selected. In a practical application, one ought to take multiple permutations. 

PA has a lot going for it. It is a simple, concrete method. It is easy to code in software, and it is already implemented in several \verb+R+ packages, including \verb+nFactors+ \citep{raiche2010package}. In addition, there is a great deal of empirical evidence that it works well, compared to other standard methods. The main other methods are Kaiser's ``eigenvalues larger than one'' rule \citep{kaiser1960application}, Bartlett's likelihood ratio test \citep{bartlett1950tests}, the
scree plot \citep{cattell1966scree}, and Velicer's minimum-average-partial criterion (MAP) \citep{velicer1976determining}. Empirical evidence favors PA. In an extensive simulation study, \cite{zwic:veli:1986} concluded that PA and MAP are consistently accurate.  \cite{peres2005many} compared  20 methods for selecting the number of components in factor analysis. They found the PA and its variants are the best methods.  \cite{owen:wang:2016} proposed a bi-cross-validation method, focusing on estimating the factor component. Even  for this new objective function, PA was one of the most accurate methods.%, while the authors' own method outperformed it.

Based on these findings, PA has become a standard textbook method:  
\benum

\item \cite{brown2014confirmatory} notes that PA ``\emph{is accurate in the vast majority of cases}''
\item \cite{hayton2004factor} review evidence from social science and management that PA is ``\emph{one of the most accurate
factor retention methods}'' %They also note that PA is not that widely used.

\item \cite{costello2005best} write that PA is ``\emph{accurate and easy to use}''

\item In the context of PCA, \cite{friedman2009elements} use it as the default method for selecting the number of significant components (see Fig 14.24, p. 538).

\eenum

While PA has not been employed enough by practitioners \citep{hayton2004factor,gaskin2014exploratory}, recently it has started to become more widely used by leading researchers in applied statistics, especially in the biological sciences: 

\benum
\item   \cite{leek2008general} use it in a general framework for multiple testing dependence. It is the default method in the popular \verb|sva| package for "Surrogate Variable Analysis"  \citep{leek2007capturing}.

%\item  \cite{quadeer2014statistical} use it to select the number of components used for clustering in a virology application.

\item Wing H. Wong's group used it to select the number of components when performing dimension reduction while controlling for covariates \citep{lin2016simultaneous}.

\item  \cite{gerard2017unifying} use it in their general methodology for removing unwanted variation (RUV) based on negative controls.

\item  Zhou, Marron and Wright use a block permutation version in eigenvalue significance tests for genetic association \citep{zhou2017eigenvalue}.

\eenum

PA is not the end of the story, and there are newer methods \citep[see e.g.,][]{kritchman2008determining, onatski2012asymptotics,josse2012selecting,gaskin2014exploratory}. However PA is by far the most popular method, and thus it is worth studying.

\subsection{The lack of theory, and this work}

Despite this empirical success, there is essentially no theoretical justification that PA works. For instance, \cite{green2012proposed} calls PA ``\emph{at best a heuristic approach rather than a mathematically rigorous one}''. Clearly, this lack of understanding limits the appeal of PA. The perceived lack of rigor prevents practitioners from making the best decision on which method to use.

In this paper, we will develop the first fully rigorous understanding of PA. We will show that PA selects the large factors in a broad range of factor models. Importantly, PA selects only the factors whose size is above a certain well-specified threshold. The key requirements are that (1) the dimension $p$ is large compared to the sample size $n$, and (2) each factor loads on more than just a few variables. These are quantified into precise mathematical statements. See Thm, \ref{fa_cons} for a clean result. 

%Fortunately, the new results turn out to be mathematically elegant. 
The basic idea is simple: The factor model can be written in a signal-plus-noise form $X = S+N$, where $S$ is of low rank. A random permutation ``destroys'' the signal $S$, transforming it into a noise-like matrix (see Fig. \ref{fig:sig_des});  while keeping the noise distribution invariant. This allows the identification of the factors above noise level. 

We hope that our results will de-mystify PA, and help practitioners understand when it is the most suitable method in factor analysis and PCA. We also hope that the mathematical approach developed in this paper will become useful to improve PA, as discussed at the end of the paper. In a follow-up work, we have been able to address several of the limitations of PA \citep{dobriban2017deterministic}.

Roughly speaking, our contributions are as follows:
\benum
\item We provide an asymptotic analysis of PA in ``low-rank-signal plus noise'' models (Sec. \ref{gen_app}). We prove a basic Consistency Lemma (Lem. \ref{cons_lem}) giving general conditions on the signal and the noise when PA selects the large factors, which we call \emph{perceptible}.
\item We then provide concrete assumptions under which the general conditions for the signal (Sec. \ref{sig_mod}) and the noise (Sec. \ref{noise_mod}) hold. This involves new bounds on operator norms of permutation random matrices (Thm. \ref{sig_cond}). 
\item We apply these results to show that PA selects the perceptible factors in factor models (Sec. \ref{simple_app}). For pedagogical reasons, this is presented \emph{before} the general signal-plus-noise approach.

\item We discuss the application of PA in PCA (Sec. \ref{spiked_mod}). We provide numerical evidence supporting our claims (Sec. \ref{nume}), which are all reproducible with software available at \url{github.com/dobriban/PA}. We close with a discussion of future work (Sec. \ref{conc}).
%\item 
\eenum

\section{Consistency of permutation methods (PA)}

\subsection{A simple result}
\label{simple_app}
We start by presenting a simple consistency result for PA. Recall that we have observations $x_{ij}$ following the standard factor model   \eqref{feq}.  
The vector $x_i = (x_{i1},\ldots,x_{ip})^\top$ of observations for the $i$-th sample can be expressed as 
$$x_i = \Lambda \eta_i +\ep_i,$$
where $\Lambda$ is the $p\times r$ factor loading matrix with entries $\lambda_{jk}$, $\eta_i $ is the $r$-vector of factor values for the $i$-th sample, and $\ep_i$ is the $p$-vector of unique factors.  The factors $\eta_i$ are random variables, while the loadings $\Lambda$ are fixed parameters.
The $n \times p$ matrix $X = (x_1,\ldots,x_n)^\top$ can be written as 
$$X = H\Lambda^\top +\mathcal{E}.$$
Here $H$ is the $n \times r$ matrix containing the factor values $\eta_{ij}$, and $\mathcal{E}$ is the $n \times p$ matrix containing the noise $\ep_{ij}$.  The covariance matrix of one sample $x_i$ is 
$$\Sigma =  \Lambda \Psi  \Lambda^\top +\Phi,$$
where $\Phi = \diag(\Phi_i)$ is the diagonal matrix of idiosyncratic variances.
%Here $\Sigma$ is a correlation matrix, so $\diag \Sigma = I_p$. 

It is well known that the parameters are not uniquely identified in this model \citep{anderson1958introduction}. It turns out, however, that the number of \emph{large factors} is asymptotically well defined. The key is to quantify \emph{the size of the noise} via the operator norm of the noise matrix. For this, we consider a sequence of factor models where the sample size $n$ or the dimension $p$ grows. In this asymptotic setting, we suppose that we can re-normalize the data so that 
$$c_{n,p}^{-1}\|\mathcal{E}\| \to b>0$$ almost surely (a.s.), or in probability. Here $\|M\|$ denotes the operator norm, the largest singular value of a matrix $M$, and $c_{n,p}$ are some constants. We define $b$ as \emph{the size of the noise}. As we will see, there are many settings where $c_{n,p}$ exists and thus $b$ is well defined. Convergence a.s. and in probability are both considered, and allow for  ``parallel'' theories. 

We define the \emph{above-noise factors} as those whose contribution to variation in the data is larger than the size of the noise. We measure the contribution of the $k$-th factor by the $k$-th largest singular value $\sigma_k(X)$. We say that the $k$-th factor is above-noise  if $c_{n,p}^{-1}\sigma_k(X)>b$ a.s, (or in probability). It may seem surprising that this definition depends on the entire asymptotic setting, and not just on finite values of $n,p$.  However, since our entire approach is asymptotic, this is reasonable. In practice, a factor is above-noise if its effect on variation is larger than the noise. 

In addition, it will be useful to define \emph{perceptible factors}, whose effect on variation differs from the size of the noise by some definite value.  We define \emph{perceptible factors} as those indices $k$ for which $c_{n,p}^{-1}\sigma_k(X)>b+\ep$ a.s (or in probability) for some $\ep>0$. Similarly, we define \emph{imperceptible factors} as those indices $k$ for which $c_{n,p}^{-1}\sigma_k(X)<b-\ep$  for some $\ep>0$.

Our main result is that PA selects the perceptible factors. To state this we will need the \emph{distribution function} of  $\Phi$, the discrete probability mass function placing equal mass on all $\Phi_i$.  For a bounded probability distribution $H$, we will also need the quantity $s(H)=\sup \supp(H)$, the supremum of the support of $H$. For a discrete distribution $H$ taking values on $h_1<h_2<\ldots<h_l$, we have $s(H)=\max_i h_i = h_l$. Let $\Psi^{1/2}$ be the symmetric square root of $\Psi$, and let us define the scaled factor loading matrix $\Lambda\Psi^{1/2} = [f_1,\ldots,f_r]$.

\begin{theorem}[Parallel analysis selects the perceptible factors]
\label{fa_cons}
Suppose we observe $n$ independent samples from the $p$-dimensional factor model $x_i = \Lambda \eta_i +\ep_i$. Assume the following conditions:

\benum

\item {\bf Factors}: The number $r$ of factors is fixed. The factors $\eta_i$ have the form $\eta_i  = \Psi^{1/2} U_i$, where $U_i$ have $r$ independent  entries with mean zero and variance 1.
\item {\bf Idiosyncratic terms}:  The idiosyncratic terms are $\ep_i=\Phi^{1/2} Z_i$, where $\Phi^{1/2}$ is a diagonal matrix, and $Z_i$ have $p$ independent entries with mean zero, variance one, and bounded fourth moment. %Thus the total covariance is $\Sigma =  \Lambda \Psi\Lambda^\top +\Phi$.
\item  {\bf Asymptotics}: $n,p \to \infty$ such that one of the following conditions holds: 

\benum
\item $p/n \to \gamma >0$, while the distribution function of  $\Phi$ converges weakly to $H$, and $\max \Phi_j \to s(H)$.  The entries of $Z_i$ have bounded 6+$\ep$-th moment.

\item $p/n \to \infty$, while the entries $\Phi_j \le C \tr[\Phi]/p$ for all $j$. %Here $n$ can be fixed or grow much slower than $p$. 

\eenum

\item {\bf  Factor loadings}: The $r$ vectors of scaled loadings $f_k$ are each bounded, in the sense that $|f_k|_2 \le C n^{1/4-\delta/2}$. They are also delocalized, in the sense that $|f_k|_4/|f_k|_2 \to 0$. 

\eenum 

Then with probability tending to one, \emph{parallel analysis selects all perceptible factors}, and \emph{no separated imperceptible factors}.

\end{theorem}
The proof of the theorem follows from the more general approach developed below, and is given later in Sec. \ref{fa_cons_pf}. 

Importantly, PA selects only the sufficiently large factors, whose size is above a certain well-specified threshold. While in general there we are not aware of a simple description of how large the factors must be to be selected, in the special case of spiked models, the thresholds become much more explicit, see Corollary \ref{PA_spiked}.

The theory allows growing factors, but only at the rate $|f_k|_2^2 \sim n^{1/2-\delta}$. In econometics, \citep[e.g.,][etc]{bai2008large,onatski2009testing}, the factors grow linearly at rate $n$, so the current theory is weaker. However, PA is actually used more in computational genomics and social science,  where the factors are typically much weaker. So, we think that the current theoretical results are a good first step. In future work it would be important to examine the issue of strong factors in more detail.

Assumption 3(a) is somewhat technical. We assume that the distribution function of  $\Phi$ converges weakly to a certain limit probability distribution $H$. This means that there is a certain ``regularity'' in the overall distribution of the variances. Mathematically, it is a standard assumption in random matrix theory \citep{bai2009spectral, yao2015large}. This and $\max \Phi_j \to s(H)$ are technical assumptions needed to guarantee that the \emph{size of the noise} $b>0$ is well defined. %In practice they mean that the distribution of idiosyncratic variances is ``regular''.

The conclusion is that under reasonable conditions, PA selects the perceptible factors with high probability. A key feature of the theorem is that it allows both the sample size $n$ and the dimension $p$ to be large. Both $p/n \to \gamma >0$ and $p/n \to \infty$ are handled, so  that $p$ can be much larger than $n$. However $p/n\to 0$ is not handled, and we will see in simulations that PA does not always work in this regime. This is one of the main conclusions of the current paper: \emph{PA works well when the dimension is large.} This can be interpreted as a blessing of dimensionality.

The intuitive explanation is that when $p$ is small, the factor loadings increase the effective variance of the features of the data. Thus, when the features are permuted, the variances are overestimated. Hence, the true noise level is overestimated, and some smaller factors may not be detected. However, this heuristic argument at least indicates that PA will likely be \emph{conservative} in this case.

A second key feature is that the factor loading vectors $\lambda_j$ need  to load on more than just a few variables. Suppose for simplicity that we have an \emph{orthogonal factor model}, $\Psi = I_r$, so $f_k = \lambda_k$. The formal requirement is that the ratio of the $\ell_4$ and $\ell_2$ norms vanishes: $|\lambda_j|_4/|\lambda_j|_2 \to 0$. For instance, $\lambda_j = (1,0,0,\ldots)$ does not satisfy this, but $\lambda_j = p^{-1/2}(1,1,$ $\ldots)$ does. In fact, $\lambda_j$ can have nonzero loadings on a vanishing fraction $\delta$ of the entries, as long as $n\delta\to\infty$, and the entry sizes are lower bounded.  An interesting example is a \emph{clustering} pattern, where the $\lambda_{jk} = |S_j|^{-1/2} I(k\in S_j)$, and $S_j$ are mutually disjoint clusters with sizes $|S_j|\to \infty$\footnote{We thank Jingshu Wang for this example.}.

Thus, our assumptions are not restrictive. In practice, they mean that the loadings cannot concentrate on a small number of variables.  This assumption is similar to non-sparsity, delocalization, or incoherence conditions seen in other works. This is the second main conclusion of the current paper: \emph{PA works well when the factors load on more than just a few variables.} 

In summary, our main conclusion is that PA works well when:

\benum

\item the dimension of the data is large, and
\item the factors load on more than just a few variables
\eenum

Finally, this result concerns only separated factors, for which $c_{n,p}^{-1} \sigma_k(X) > b+\ep$ or $c_{n,p}^{-1} \sigma_k(X) < b-\ep$, but not factors near the noise level. Intuitively, these latter are ``hard to distinguish'' from the noise. This is related to the difficulty of understanding the critical regime of spiked models \citep[e.g.,][]{yao2015large}. At the moment, we do not have a clear understanding of PA in the critical regime. 

In addition, we note that the new theoretical guarantees for PA cover roughly the same regime as when some of the existing methods based on eigenvalues are known to work \citep{kritchman2008determining, onatski2012asymptotics}. However, PA is a very popular method, used widely in science, and thus it is important to understand what it does.  

\subsection{The general approach: signal-plus-noise matrices}
\label{gen_app}
We now shift to a more general approach, which will be used in the rest of the paper. We will work with \emph{signal-plus-noise} matrices $X = S+N$, where $S$ is an $n\times p$ signal matrix of low rank, and $N$ is an $n\times p$ noise matrix. In the standard factor model, $X = H\Lambda^\top +\mathcal{E}$. The first term is the signal due to the factor component, whose rank is at most $r$. Thus the factor model falls into the signal-plus-noise framework.

PA works with the permuted matrices $X_\pi$. By linearity, $\pi$ acts separately on $S$ and $N$, so $X_\pi = S_\pi+N_\pi$.  Intuitively, permutations keep the noise distribution unchanged, while ``destroying'' the signal. Think of $S$ as a ``smooth'' matrix, which can be achieved by reordering rows and columns. A typical permutation $\pi$ transforms this into a ``rough'', ``noise-like'' matrix $S_\pi$. See Figure \ref{fig:sig_des}. This has the same entries as $S$, but is typically of full rank. While the Frobenius norm (sum of squared entries) is preserved, the operator norm can be dramatically reduced. Symbolically: 
$$X_\pi = S_\pi+N_\pi \approx N_\pi.$$
Therefore, $X_\pi$ behaves like the noise $N_\pi$. If the noise is sufficiently "invariant under permutations", then it may be possible to estimate its "size". Write $X=_dY$ if the random variables $X,Y$ have the same distribution, and suppose that  $N_\pi=_d N$. Then from the previous two equations, 
$$\|X_\pi\| \approx_d \|N\|.$$
Thus, the operator norms of the permuted matrix $X_\pi$ and the noise matrix are roughly the same. Selecting factors whose singular values are larger than $|X_{\pi}|$ is roughly the same as comparing to the operator norm of the noise. This provides an intuitive justification that PA selects the perceptible factors, as defined above. The rest of the paper makes this intuition precise. 

\subsubsection{The consistency lemma}

The first step is to formalize the above argument into a rigorous \emph{consistency lemma}. This is a general result that gives broad conditions for the \emph{signal} and the \emph{noise} under which PA is consistent. We will then give examples when the two sets of conditions hold. 

In the signal-plus-noise model $X=S+N$, suppose $S$ is deterministic and $N$ is random; this can be achieved by conditioning on $S$. We want to provide a result that holds under a variety of asymptotic settings. Thus, consider an asymptotic setting $\mathcal{A}$, for instance 
\benum

\item Classical asymptotics, where $n \to \infty$ and $p$ is fixed

\item Proportional-limit asymptotics, where $n,p \to \infty$, such that $p/n\to\gamma>0$. This is also known as high-dimensional asymptotics, random matrix asymptotics, or the thermodynamic limit  \citep[e.g.,][]{paul2014random, yao2015large, dobriban2015high}.

\item ``Big $n$, bigger $p$'' asymptotics, where $n,p \to \infty$, such that $p/n\to\infty$, 

%\item 
%Fixed-sample asymptotics, where $p \to \infty$, $n$ fixed.

\eenum

Our consistency lemma does not depend on the specific type of asymptotics. It applies to all of the above settings. 

Next, fix a mode of stochastic convergence, either convergence almost surely (a.s.), or in probability. Below we will use a.s., but the equivalent results hold for in probability, \emph{mutatis mutandis}. We will use both in the application to factor models. 

In the asymptotic setting $\mathcal{A}$, suppose that the signal matrix belongs to some parameter space $S \in \Theta$, and the noise has some distribution $N\sim P_N$. Suppose that we have re-normalized the data such that under $\mathcal{A}$, $$\|N\| \to b>0$$ a.s. As in the special case of factor models discussed above, we define the \emph{above-noise factors} as those indices $k$ for which $\sigma_k(X)>b$ a.s. 

Consider also a distribution of permutation arrays $\Pi$, defined for all $n,p$ of interest; for instance each permutation $\pi_j$ sampled independently and uniformly from the set of all permutations of $[n] = \{1,\ldots,n\}$. 

Before turning to usual PA, it is pedagogically helpful to define \emph{asymptotic PA} as a theoretical version of PA leading to a particularly simple analysis. Consider a finite set of permutation arrays $\Pi_0$ sampled independently, each according to $\Pi$. Asymptotic PA selects those factors for which $\sigma_k(X)>\max_{\pi \in \Pi_0} \|X_\pi\|$ a.s. This definition depends on the entire asymptotic setting $\A$, not just on finite values of $n,p$. In finite samples, we instead select the factors for which $\sigma_k(X)>\max_{\pi \in \Pi_0} \|X_\pi\|$. Asymptotic PA is an ``oracle method'', but it leads to very elegant results. The second definition is practically feasible, and we will see that the results are still nice. 

As we will see, in our setting selecting the factors above the 95th percentile of $\{\|X_\pi\|:\pi\in\Pi_0\}$ leads to an entirely equivalent method. This is because the values $\|X_\pi\|, \pi\in\Pi_0$ all converge a.s. to the same value. The difference is only important for the properties of PA as a hypothesis testing method, specifically its control of the type I error. Thus, we focus on asymptotic PA as defined above: 

\begin{lemma}[Consistency lemma]
\label{cons_lem}
Suppose the following

\benum
\item{\bf Noise invariance}: The distribution of the noise is invariant under permutations, so $N =_d N_\pi$, where the equality in distribution is taken with respect to the joint randomness of the noise matrix  $N\sim P_N$ and the independently chosen permutation $\pi \sim \Pi$. 

\item {\bf Signal destruction}: Under the asymptotics $\mathcal{A}$, we have $\|S_\pi\| \to 0$ a.s., for all $S \in \Theta$, where the randomness is induced by the random permutation  $\pi \sim \Pi$. 

\eenum

Then, asymptotic parallel analysis is consistent for selecting the above-noise factors. 
\end{lemma}

\begin{proof}
Since $X_\pi = S_\pi+N_\pi$, by the triangle inequality we have $|\|X_\pi\| - \|N_\pi\|| \le \|S_\pi\| \to 0$. Now, by invariance $N =_d N_\pi$, and by the convergence $\|N\| \to b$ to the noise level, we have that the operator norms of the permuted matrices also converge: $\|N_\pi\| =_d \|N\| \to b$.  Hence, it follows that $\|X_\pi\| \to b$.

Asymptotic parallel analysis selects the factors for which $\sigma_k(X)>\max_{\pi \in \Pi_0}$ $ \|X_\pi\|$ a.s. Since $\Pi_0$ is of fixed size, based on the above argument, this is the same as those factors for which $\sigma_k(X)>b$ a.s.,  which are exactly the above-noise factors. This finishes the proof. 

\end{proof}

This result is a very elegant form of the statement that PA selects the number of above-noise factors. However, it deals with asymptotic PA, which is an oracle method only defined asymptotically. Can we remove the asymptotics from the definition of the method?

Recall that we consider the version of non-asymptotic parallel analysis which selects the factors for which $\sigma_k(X)>\max_{\pi \in \Pi_0} \|X_\pi\|$. %Note that this is a \emph{bona fide} method, well defined for all fixed $n,p$.
Since above-noise factors are defined asymptotically by comparing $\sigma_k(X)$ to $b$, and non-asymptotic PA depends only on finite $n,p$, it is not clear how to show that PA selects all above-noise factors. However, this becomes clear if we focus on \emph{separated above-noise factors}, called \emph{perceptible factors}, as indicated previously. Thus, we define \emph{perceptible factors} as the $k$ for which $\sigma_k(X)>b+\ep$ a.s. for some $\ep>0$.  We also define \emph{imperceptible factors} as the $k$ for which $\sigma_k(X)<b-\ep$ a.s. for some $\ep>0$.  We then have the following analogue of the previous lemma: 

\begin{lemma}[Consistency lemma for non-asymptotic PA]
\label{non_asy_pa}
Under the conditions of Lemma \ref{cons_lem}, PA selects all perceptible factors, and no imperceptible factors, a.s.

\end{lemma}

\begin{proof}

Non-asymptotic parallel analysis selects the factors for which $\sigma_k(X)$ $>\max_{\pi \in \Pi_0} \|X_\pi\|$. Since  $\max_{\pi \in \Pi_0} \|X_\pi\| \to b$ a.s., it is immediate that this includes all perceptible factors, and no imperceptible factors, a.s.

\end{proof}

These results give broad conditions for the signal and the noise under which PA is consistent. The real work is always to show that these conditions hold in particular cases of interest, such as for factor models. 

\subsubsection{Conditions on the signal and the noise}

When do our assumptions hold? We start here with a brief discussion. 

For the noise, we  need two assumptions. 
\benum
\item  The existence of a well-defined asymptotic noise level $b>0$ such that $\|N\| \to b>0$. This imposes a restriction on the noise models. For this condition, it will be helpful that operator norms of random matrices have been studied thoroughly, and thus there are broad conditions under which such convergence is known. 
\item  The invariance of the distribution of noise to permutations: $N =_d N_\pi$. There is a tradeoff: the more general the noise distribution, the smaller the set of permutations that keeps it invariant. Thus, this also imposes a restriction, because we may need a large set of permutations to cancel out the signal terms, as we see next.
\eenum

For the signal, we need one assumption:

\benum
\item  The operator norm of the permuted signal matrices vanishes, $\|S_\pi\| \to 0$ a.s. for all $S \in \Theta$. There is a tradeoff here too: The larger the parameter space $\Theta$, the harder this is, and the more permutations are needed to get enough ``averaging'' for this to hold. For certain signals, e.g., the all ones matrix with $S_{ij}=1$, this is entirely impossible, because every permutation keeps the matrix unchanged. 

%An important aspect of this condition is that it allows factors of ``growing'' strength. Specifically, the operator norm of the signal matrix $S$, before normalization, can grow

\eenum

In the next two sections, we examine the two sets of conditions in more detail. 
%[More detail]

%[Also, finite sample guarantees, like level]

\section{Signal models}
\label{sig_mod}

When do our assumptions on the signal hold? We need that permutations ``destroy'' the signal structure, so that $\|S_\pi\| \to 0$ a.s., for all $S \in \Theta$. Consider a rank one signal matrix $S = uv^\top$. Then, acting on this by a permutation array $\pi$ we get (denoting by $\odot$ elementwise product of matrices):
$$S_\pi = (uv^\top)_\pi = (u1^\top)_\pi \odot  1v^\top = [\pi_1(u); \pi_2(u); \ldots; \pi_p(u)] \odot 1v^\top.$$
Each permutation $\pi_j$ permutes the corresponding column $j$ of $S$. This column equals $v_j u$, so $\pi_j$ permutes the entries of $u$. Since $\pi_j$ is a uniformly random permutation, the distribution of this column is uniform on all permutations of $u$, and is ``modulated'' by $v_j$. If the entries of $u$ sum to zero, this is effectively ``noise'' of variance approximately $v_j^2/n$. The $n$ entries of column $j$ are exchangeable random variables, which are almost independent for large $n$. Hence, $(uv^\top)_\pi $ is a random matrix whose columns are independent, and within each column the entries are nearly independent, with variance approximately $v_j^2/n$. If the entries of the matrix were independent, we could use well-known results controlling its operator norm \citep[e.g.,][]{vershynin2010introduction}. However, since the entries are dependent, we need to establish these results here from first principles. 

 A first simplification is that we can separate the component corresponding to $u \in span(1)$, where $1 = (1,1,\ldots)^\top$ is the all ones vector, and its orthocomplement. The first part is kept invariant by the permutation. So we just need to assume that it goes to zero. Let $\theta \cdot n^{-1/2} 1 \cdot v^\top$ be this component, where $|u|_2=|v|_2=1$. Then, we need to assume that $\theta \to 0$, because we need
$$
\|\pi(\theta \cdot n^{-1/2} 1 \cdot v^\top)\| 
=\|\theta\cdot n^{-1/2} 1 \cdot v^\top\| 
=\theta |v|_2 = \theta \to 0.
$$
In our application to factor models, we will separate this component, and show that $\theta \to 0$ holds. 

On the orthocomplement, we will use the moment method to show $\|S_\pi\|\to 0$.  We have that $\|A\|^k \le \tr (A^\top A)^k $ for all $k$. Hence,
$$P(\|A\|>\ep)  =  P(\|A\|^k >\ep^k) \le P( \tr(A^\top A)^k >\ep^k) \le \ep^{-k} \E\tr(A^\top A)^k$$
Thus, to show that $\|S_\pi\| \to 0$ in probability, it is enough to argue that $\E \tr (A^\top A)^k  \to 0$ for some $k>0$. To show a.s. convergence, by the Borel-Cantelli lemma we need that $\E\tr(A^\top A)^k$ is summable for some $k>0$. After the appropriate moment calculations, we obtain the following result:

\begin{theorem}[Requirements on the signal]
\label{sig_cond}
Consider signals of the form $S = n^{-1/2}\theta\cdot 1v_0^\top +T$, where $T=\sum_{i=1}^r \theta_i u_i v_i^\top$, with $|u_i|_2=|v_i|_2=1$, and $u_i^\top 1 = 0$ for all $i$.  Here $r$ can be fixed or change with the dimensions $n,p$. Suppose that $\theta \to 0$. Define the constants $A_{nk}$ for $k=2,3,4$ as 

$$A_{nk} = \sum_{i=1}^r \theta_i \cdot C_{k}(v_i)^{1/(2k)}$$

where $C_{k}(v)$ are defined as 

\benum
\item $C_{2}(v) = 1/(n-1) +|v|_4^4$
\item $C_{3}(v) = 1/(n-1)^2+9n^{-1} |v|_4^4 + |v|_6^6$
\item $C_{4}(v) = 1/(n-1)^3 +4/(n-1)^2 |v|_4^4 + 12n^{-1} [|v|_4^8 + |v|_6^6]+ |v|_8^8$

\eenum
Then, 
$$[\E\tr (S_\pi^\top S_\pi)^k]^{1/(2k)} \le A_{nk}.$$ 

Therefore, 

\benum
\item If $A_{nk}\to 0$, then $\|S_\pi\| \to 0$ in probability. 
\item If  $A_{nk}^{2k}$ are summable, then $\|S_\pi\| \to 0$ almost surely. 
\eenum

\end{theorem}

The proof is provided in Sec. \ref{sig_cond_pf}. Note that the second condition can only hold for $k\ge 3$, because $n^{-1}\lesssim A_{n2}^4$. 

An interesting consequence of this result is that PA works in certain cases even when the \emph{number} of signals as well as the \emph{strength} of signals grows to infinity simultaneously. Indeed, suppose $v_i$ are all maximally delocalized, so that $|v_i|_\infty \le C p^{-1/2}$ for some $C>0$. Then, we have $|v_i|_4^4 \le C^4 p^{-1}$, and $|v_i|_6^6 \le C^6 p^{-2}$, therefore $C_3(v_i) \le C'(n^{-2}+p^{-2})$ and

$$A_{n3} ^6\le C' \left[\sum_{i=1}^r \theta_i \right]^{6} \cdot \left ( \frac{1}{p^2}+\frac{1}{n^2}\right)$$

So we need to find conditions under which this goes to zero or is summable. When $n,p\to \infty$, for $A_{n3} \to 0$ it is enough that $\sum_{i=1}^r \theta_i = O(\min(n,p)^{1/3-\ep})$ for some $\ep>0$. For instance, when $p/n\to\gamma>0$, is enough that $m|\theta|_{\infty} =O(n^{1/3-\ep})$. We can take $n^{1/6}$ spikes (signals) of size $n^{1/6-\ep}$ each, and parallel analysis will work. Alternatively, we can take one spike of strength $\theta = n^{1/3-\ep}$.

This is important, because it allows us to handle two seemingly very different regimes simultaneously: the ``explosive'', i.e., growing, spikes of the type that are common in econometrics \citep{bai2008large}, while also handling the constant-sized spikes that are common in the literature in random matrix theory and applications to statistics \citep[e.g.,][]{paul2014random, yao2015large}.

%The above discussion suggests that PA works well in a broad regime when the dimension $p$ is large. In contrast, increasing the sample size $n$ does not bring obvious performance improvements. In fact, we will see in simulation studies that the benefit due to increased sample size can be minimal in certain cases. This observation presents a marked contrast to the usual setting of classical asymptotic statistics, where increasing the sample size leads to improved performance, while increasing the dimension leads to a decreased performance. For this reason it is perhaps not an exaggeration to say that the workings of PA are \emph{a blessing of dimensionality}.

\subsection{Optimality considerations}
\label{optim}
\subsubsection{Signal strength}
\label{optim_sig}
The above results and discussion show that PA selects the perceptible factors in models of the form $\theta u v^\top +N$ as long as the signal strengh $\theta$ is not too large. For instance, we saw that $\theta = \min(n,p)^{1/3-\ep}$ suffices for delocalized signals. It may seem counterintuitive that a strong signal can decrease the performance of PA. Is this a weakness of our theoretical analysis, or a weakness of the method? 

To understand this issue, we recall that PA ``transforms the signal into noise''. Thus, a large signal is transformed into large noise, which can lead to the overestimation of the true noise level. In turn, this may prevent the selection of the above-noise factors which are \emph{not} above the estimated noise level. So the problem is with PA, not with our result.

More precisely, the permuted matrix $S_\pi=(\theta u v^\top)_\pi$ is a matrix with independent columns, and within each column, with approximately independent (in truth, exchangeable) bounded entries. If the entries were independent with variance $\sigma^2/n$, the operator norm would be of order $\sigma \cdot [1+(p/n)^{1/2}]$ \citep{bai2009spectral,vershynin2010introduction}. In our case, $\tr S_\pi^\top S_\pi = \tr S^\top S = \theta^2$, so heuristically, $p \sigma^2 \approx \theta^2$. Thus, heuristically
$$\|S_\pi\| \approx  \theta \cdot [n^{-1/2}+p^{-1/2}].$$
In our consistency lemma, we showed that PA will select the above-noise factors if $\|S_\pi\| \to 0$, which amounts to $\theta \cdot [n^{-1/2}+p^{-1/2}] \to 0$. In particular, under high-dimensional asymptotics when $p/n\to \gamma>0$, this holds if $\theta/n^{1/2} \to 0$. This suggests that our result $\theta = n^{1/3-\ep}$ is \emph{not} optimal, and PA works more broadly. We note that a $k$-th moment bound in Theorem \ref{sig_cond} will allow $\theta = o(p^{1/2-1/(2k)})$. However, much more work is needed to show such a bound.  In principle, the current moment calculations should work, but this is much beyond the scope of this work, as they become too hard for large $k$. 

Thus it appears that very strong factors lead to problems with PA. This is counter-intuitive, because strong factors should be easy to detect. However, this apparent paradox can be resolved. The noise level estimated by PA is of the order of 
$$f_{est} \approx \max(b, \sum_{k=1}^r\theta_k \cdot [n^{-1/2}+p^{-1/2}]).$$
A factors is not selected if $\sigma_k(X) < f_{est}$. From the analysis of spiked covariance matrix models, when the noise is of the form $n^{-1/2}X$ for $X$ with iid entries, we expect the empirical singular values to behave (very roughly speaking) like $\sigma_k(X) \approx \theta_k+(p/n)^{1/2} $. From these two approximations, a factor $k$ is not selected only if $\theta_k/\sum_{k=1}^r\theta_k \lesssim n^{-1/2}+p^{-1/2}$. This shows that only the \emph{relatively unimportant} factors are not selected by PA, in the sense that the relative strength of the factor $k$, $\theta_k/\sum_{k=1}^r\theta_k$, must be small. In this sense, PA still selects the "relatively large" factors.

\subsubsection{Delocalization}

What is the precise condition needed on $v$?  In  Theorem \ref{sig_cond} we gave several conditions depending on the norms $|v|_k$, for $k=4,6,8$, which all amount to some form of delocalization, in the sense that $v$ is non-sparse. Some non-sparsity condition is needed. Indeed, when $v = (1,0,\ldots,0)$,  then a permutation only acts on the first column of $u v^\top$, thus the operator norm is unchanged. Some form of delocalization is needed, but finding the precise condition may need a new theoretical approach, which is beyond our current scope.  

%A second way to approach the problem is to realize that PA is 

\section{Noise models}
\label{noise_mod}
When do our assumptions on the noise  hold? We need two assumptions: invariance to permutations, and operator norm convergence. 

\subsection{Invariance}
\label{invar}

We need the that noise is invariant in distribution to permutations $N =_d N_\pi$, where the permutation $\pi$ is also random. We will study when invariance holds for any \emph{fixed} permutation $\pi$; then it will also hold for random permutations $\pi \sim \Pi_0$ chosen indepenendently from $N$.  This is a non-asymptotic condition, so the findings will apply to any asymptotic setting we consider. 

For $N =_d N_\pi$ it is enough if columns of $N$ are independent, and each column has exchangeable entries. But a weaker condition is enough. Suppose for instance that $(N_{ij})_{ij}$ are an equicorrelated Gaussian random vector, in matrix form. Then clearly they are not independent, but are still invariant under any permutation. 

Following this logic, if we vectorize the matrix $N$ into an $np$-length vector, whose blocks of size $n$ are the different features indexed from $1$ to $p$, the condition $N =_d N_\pi$ means that the distribution is invariant under permutations within the blocks. For a Gaussian random vector, in terms of the covariance matrix, this means that it has the following block structure:

\bitem
\item
$\Var{N_{ij}} = \sigma_j^2$: Within any column $j$, the entries $N_{ij}$ are exchangeable random variables. Thus, they must have the same variance $\sigma_j^2$.
\item
$\Cov{N_{ij}, N_{i'j}} = \tau_j^2$ for $i \neq i'$:  Similarly, distinct entries $N_{ij}, N_{i'j}$ in the same column must have the same covariance.
\item
$\Cov{N_{ij}, N_{kj'}} = \eta_{jj'}^2$ for $j \neq j'$: Consider two distinct columns $j$, $j'$. Since the entries within each of them can be permuted independently of each other while preserving the distribution, the covariance between any two entries $N_{ij}, N_{kj'}$ must be the same. 

\eitem

Equivalently, one has the explicit representation:

$$N = \mathcal{E} D^{1/2} + 1 z^\top \Sigma^{1/2},$$

where $ \mathcal{E} $ is $n\times p$ matrix of iid Gaussians, $D$ is diagonal, $z \sim \N(0,I_p)$, and $\Sigma$ is $p\times p$ PSD.  Here $\Sigma$ induces the correlations between the different columns, and is not necessarily diagonal.  Thus each sample has the form $N_i = D^{1/2}\ep_i+ \Sigma^{1/2}z \sim \N(0,D+\Sigma)$, which is a sum of a sample-specific independent diagonal normal random vector $D^{1/2}\ep_i$, and the same normal random vector $\Sigma^{1/2}z$ added to each sample. %$N_i  = D_i + Z$, where $D_i \sim \N(0,D)$ with diagonal $D$, and $Z \sim \N(0,\Sigma)$

A bit more generally, we have the following representation for noise models invariant under permutations. The proof is immediate, and thus it is omitted.

\begin{prop}[Requirements for noise invariance] Suppose that the noise matrix $N$ has rows of the form $N_i = D^{1/2}\ep_i+ z$, where $\ep_i$ are iid across $i$, and $z$ is any random vector independent of all $\ep_i$. Suppose that $\ep_i$ have independent standardized entries. and $D^{1/2}$ is diagonal. Then, the distribution of $N$ is invariant under column permutations, i.e., $N =_d N_\pi$ for any fixed permutation array $\pi$.
\label{noise_iv}
\end{prop}

This result covers factor models, where the noise is of the form $N_i = D^{1/2}\ep_i$. The term $z$ is allowed by the theory, but it is usually not of practical interest. This common term can be viewed as the mean of $N_i$. Even if the mean is zero, this can be viewed as a common perturbation affecting all samples. However, $z$ increases the overall noise level, because the operator norm $\|N\|$ is of order $(np)^{1/2}$. Thus, the presence of  $z$  renders many previously above-noise factors to sink below the noise. The practically interesting scenarios usually have $z=0$, which can be achieved by de-meaning the data. 

In addition, there are other noise models that can be reduced to the model from Prop. \ref{noise_iv}. A prime example is correlated noise models where taking the Fourier transform, or some other known orthogonal transform in space-time, leads to independent coordinates.

For instance, in time series analysis \citep[e.g.,][]{brockwell2009time}, stationary processes can be transformed into having approximately independent coordinates by the Fourier transform. By the spectral representation theorem, every zero-mean stationary process has the representation $X_t = \int_{(-\pi,\pi]} \exp(it\nu) dZ(\nu)$, where $Z$ is an orthogonal-increment process. 

The autocovariance function can be written as $\gamma(h) =$ $  \int_{(-\pi,\pi]} \exp(it\nu)$ $ dF(\nu)$, for a distribution function $F$. 
If $\gamma$ is absolutely summable and real-valued, then the process has asymptotically uncorrelated Fourier components \citep[e.g.,][Prop. 4.5.2.]{brockwell2009time}. In particular, permutation methods are heuristically reasonable. However, making this rigorous would require us to understand what happens to the permutation distribution when we have only an approximate invariance of the noise. This is beyond our scope, but is interesting for future work (see Sec. \ref{conc}).

\subsection{Operator norm convergence}
\label{op_norm_conv}

The second condition that we need for the noise is the convergence of the operator norm: $\|N\| \to b>0$. Operator norms of random matrices have been studied for a long time, see e.g., \cite{bai2009spectral,vershynin2010introduction}. We are fortunate that we can leverage some of these results. For instance, Bai, Yin, Silverstein and others have showed convergence of the operator norm of matrices of the form $N = n^{-1/2} X T^{1/2}$, where the entries of $X$ are iid standardized random variables, and where $p,n\to\infty$  such that $p/n\to\gamma>0$. We state this result together with another one for the case $p/n\to\infty$.

\begin{prop}[Requirements for noise operator norm, partly a corollary of Cor. 6.6 in \citep{bai2009spectral}] 
\label{noise_op_norm}
Suppose that the noise matrices have the form $N = c_{p}^{-1/2} X T^{1/2}$ with $c_{p} = \tr T$,  where the entries of $X$ are independent standardized random variables with bounded fourth moment, and $T$ are diagonal positive semi-definite matrices. Suppose that $p\to\infty$, and one of the following two sets of assumptions holds: 

\benum
\item $p/n \to \gamma >0$, while the distribution function of the entries of $T$ converges weakly to a limit distribution $H$, $F_{T} \Rightarrow H$.  Moreover, the operator norm of $T$ converges to the supremum of the support of $H$, $\|T\| \to \supp \sup(H)$, and the entries of $X$ have bounded 6+$\ep$-th moment.

\item The entries of $T$ are bounded as $t_j \le C \tr[T]/p$ for all $j$, while (A) $p/n \to \infty$ or (B) $n^{2+\ep} \le p$ for some $\ep>0$. 

\eenum
Then, we have $\|N\| \to b$ for some $b>0$, in probability under (2A), and almost surely under (1) or (2B).
\end{prop}

The second statement allows $n$ fixed while $p\to\infty$, which is the ``transpose'' of classical asymptotics where $p$ is fixed and $p\to \infty$.  The proof is provided later in Sec. \ref{noise_op_norm_pf}. Combined with the conditions on noise invariance, and with the conditions on the signal, this result provides a broad set of concrete scenarios when PA selects the perceptible factors. 

\section{PCA and spiked models}
\label{spiked_mod}

Should we select the number of components in PCA using PA? As \cite{jolliffe2002principal} clearly explains, there is a substantial difference between PCA and FA, and "\emph{it is usually the case that the number of components needed to achieve the objectives of PCA is greater than the number of factors in a FA of the same data}". %Therefore, our results should not be considered as shedding light on the suitability of PA for PCA. 

However, we can understand the behavior of PA in PCA within a certain class of popular \emph{spiked models}. Spiked models have served as a theoretical tool to understand PCA in high dimensions. There are several versions, some of them mutually exclusive, see for instance \cite{johnstone2001distribution, paul2007asymptotics, nadler2008finite, bai2012estimation, benaych2012singular, onatski2013asymptotic, nadakuditi2014optshrink}, and \cite{paul2014random, yao2015large} for more references.

An important class of signal-plus-noise spiked models was studied in \cite{benaych2012singular}. Here $X = S+N$, where $S = \sum_{i=1}^r \theta_i u_i v_i^\top$, and $n^{1/2} u_i, p^{1/2}v_i$ are each iid vectors with iid entries from a distribution that satisfies a log-Sobolev inequality. It is assumed that $n,p\to\infty$ such that $p/n\to\gamma>0$,  the spectral distribution of $N$ converges to a compactly supported distribution, and the top and bottom singular values converge to the respective edges $a<b$ of the distribution. The rank $r$ and the spike strengths $\theta_i$ are fixed constants. Under these conditions, \cite{benaych2012singular} derive the asymptotic limits of the empirical singular values of $X$. They establish the \emph{BBP phase transition} phenomenon discovered earlier by \cite{baik2005phase} in a special case. For $\theta_i$ above a critical value, the corresponding empirical spike $\sigma_i(X)$ will converge to a definite value larger than $b$. In this case $\theta_i$ is said to be \emph{above the phase transition}. These correspond to the perceptible factors in our terminology.  For $\theta_i$ below the critical value, $\sigma_i(X) \to b$.

Our assumptions are neither more general, nor more specific. Indeed, we allow $p/n\to\infty$ and diverging spikes, while they allow a general converging spectral distribution, without requiring permutation-invariance. 

However, our assumptions do have a nontrivial intersection. %In the special case when the noise is of the form required by Prop. \ref{noise_op_norm} (1), and the signal vectors have iid entries, both sets of assumptions hold. In this case it turns out that PA selects all spikes  above the phase transition (corresponding to the perceptible factors). However, the analysis for the spikes below the transition is more delicate, and our results do not address it. 
We can state the conclusion as a corollary. This justifies the use of permutation methods in PCA:

\begin{corollary}(PA in spiked models)
\label{PA_spiked}
Suppose we observe a signal-plus-noise spiked model $X=S+N$, where $S = \sum_{k=1}^r \theta_k u_k v_k^\top$, and $n^{1/2} u_k, p^{1/2}v_k$ are each iid vectors with iid entries from a distribution that satisfies a log-Sobolev inequality. Suppose that $n,p\to\infty$ such that $p/n\to\gamma>0$. Suppose that the noise matrix is of the form $N = n^{-1/2}YT^{1/2}$, where the entries of $Y$ are independent standardized random variables with bounded 6+$\ep$-th moments, and $T$ are diagonal positive semi-definite matrices such the distribution function of the entries of $T$ converges weakly to a limit distribution $H$. Suppose that the operator norm of $T$ converges to the supremum of the support of $H$, $\|T\| \to \supp \sup(H)$. 

According to  \cite{benaych2012singular}, Thm. 2.9, for $\theta_k$ \emph{above the phase transition}, when $\theta_k>\bar \theta$ for a certain $\bar\theta$,  the empirical singular values $\sigma_k(X)$ converge, $\sigma_k(X)\to \lambda_k$ a.s., for some $\lambda_k>b$, where $b>0$ is the limit $\|N\| \to b$, as guaranteed by Prop. \ref{noise_op_norm}. 

Then, \emph{parallel analysis selects all spikes above the phase transition}. 
\end{corollary}
The analysis for the spikes below the transition is more delicate, and our results do not address it. 

We also emphasize that, the threshold $\bar \theta$ above which the factors are selected becomes much more explicit. In particular, when the covariance of the noise is identity, $\bar \theta =\sqrt{\gamma}$, which is completely explicit as a function of $n$ and $p$.

\section{Numerical simulations}
\label{nume}
We perform numerical simulations to understand  the behavior of PA. We wish to understand the effect of key parameters of the factor model, including signal strength and delocalization of loadings, on the accuracy of PA.

\begin{figure}
\centering
\begin{minipage}{.5\textwidth}
  \centering
  \includegraphics[scale=0.5]{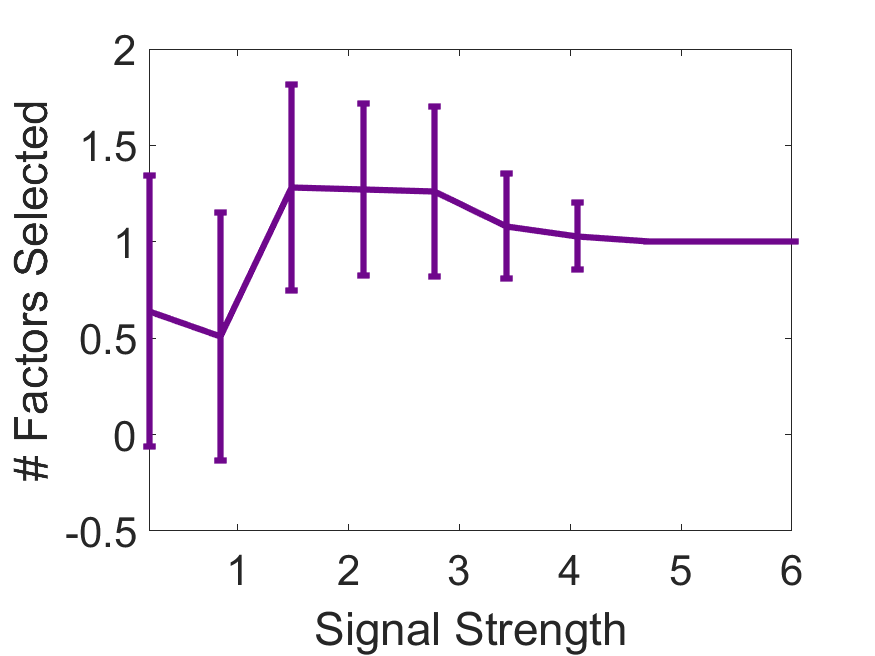}
\end{minipage}%
\begin{minipage}{.5\textwidth}
  \centering
 \includegraphics[scale=0.5]{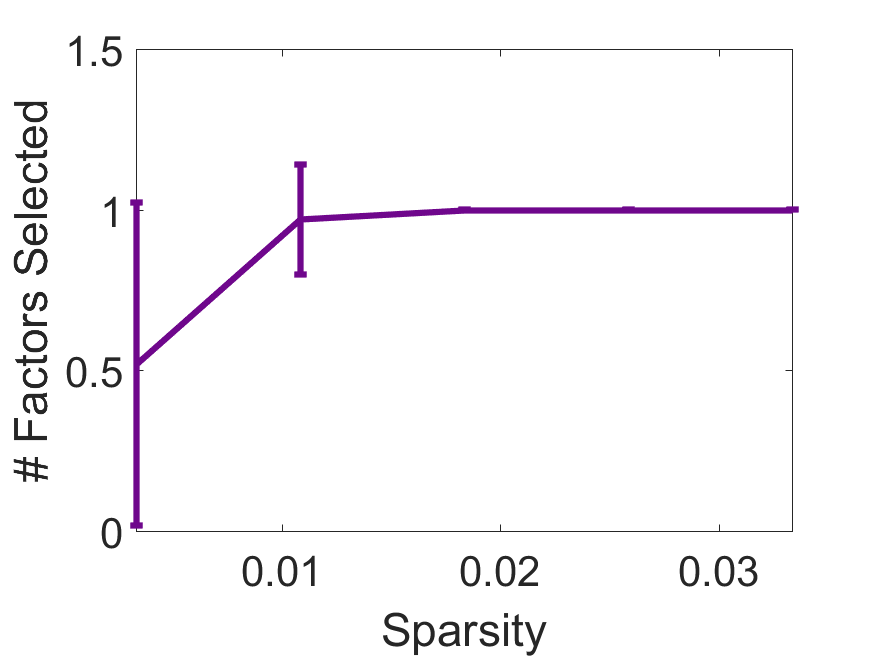}
\end{minipage}
\caption{Mean and SD of number of factors selected by PA as a function of signal strength (left), and sparsity (right).}
\label{fig:PA}
\end{figure}

\subsection{Effect of signal strength}
We simulate from the factor model $x_i = \Lambda \eta_i +\ep_i$. We generate
the noise $\ep_i \sim \N(0,I_p)$, and
the factor loadings as $\Lambda = \theta \tilde Z$, 
where $\theta>0$ is a scalar corresponding to factor strength, and $\tilde Z$ is generated by normalizing
the columns of a random matrix
$Z \sim \N(0, I_{p\times m})$.

We use a one-factor model, so $m=1$, and work with sample size $n=500$ and dimension $p=300$. It is well known that the critical regime for the signal strength $\theta$ is of the order of $\gamma^{1/2}$. We vary $\theta$ on a grid of the form $\gamma^{1/2}\cdot s$, for $s$ on a linear grid between 0.2 and 6. 

We use PA to select the number of factors. We perform 10 Monte Carlo iterations for each parameter. Motivated by our theoretical understanding, for each Monte Carlo realization of $X$, we generate only one permutation $X_\pi$. We select the first factor if $\|X\| > \|X_\pi\|$. The results in Fig. \ref{fig:PA} (a) show that PA is selects the right number of factors as soon as the signal strength $s$ is larger than $\sim 4$. This agrees with our theoretical predictions, since it shows that PA selects the perceptible factors.  

It may seem ``wrong'' that PA selects a factor even when the signal strength is nearly 0. However, this result is in agreement with our theoretical predictions. Indeed, such a factor is below-noise, but \emph{non-separated}. In line with the discussion in Sec. \ref{spiked_mod}, the singular value $\sigma_k$ corresponding to a spike below the phase transition converges to the noise level $b$. Thus, the empirical singular value does not separate from the noise level, hence PA cannot identify it as below-noise.

\subsection{Effect of delocalization}

We provide numerical evidence for our claim that ``PA works when the factors load on more than just a few variables.''
We use the same model as above. To change the delocalization of the factor scores, we define the sparsity parameter $c$, and generate $c$-sparse factors, by setting the first $\lfloor c\cdot p \rfloor$ coordinates of $Z$ to be iid Gaussians, and the remaining coordinates to be zero. One can verify that for every vector $\lambda$ of factor scores, the expected ``localization'' parameter $L=|\lambda|_4/|\lambda|_2$ is approximately $L = (9/cp)^{1/4}$, which decreases with $c$. Our theoretical results suggest that PA should select the right number of factors for ``delocalized'' or ``non-sparse'' vectors, when $c$ is large and $L$ is small.

We set $\theta = 2$ to place ourselves in a critical regime where the effect of delocalization is visible. This choice was made empirically. We vary $c$ on a grid from $1/p$ to $10/p$. We perform 100 Monte Carlo iterations for each setting of the  parameters.

The results in Fig. \ref{fig:PA} (b) show that PA tends to select the right number of factors for non-sparse, delocalized factor loadings (large $c$). This agrees with our theoretical predictions. %The results also show that PA tends to select the right number of factors \emph{on average} even when the for localized vectors. However, in this case the results of PA are much more variable. Capturing this second observation is beyond our current theory. 

It is remarkable that PA already works when the sparsity is $2\%$ ($c=0.02$). That is, if the factor loads on \emph{at least 6 out of 300 variables}, PA selects the right number of factors! This surprising result suggests that PA is likely to perform well in many realistic settings, and that delocalization is not a stringent requirement.

\subsection{Effect of dimension}

\begin{figure}
\centering
\begin{minipage}{.33\textwidth}
  \centering
  \includegraphics[scale=0.35]{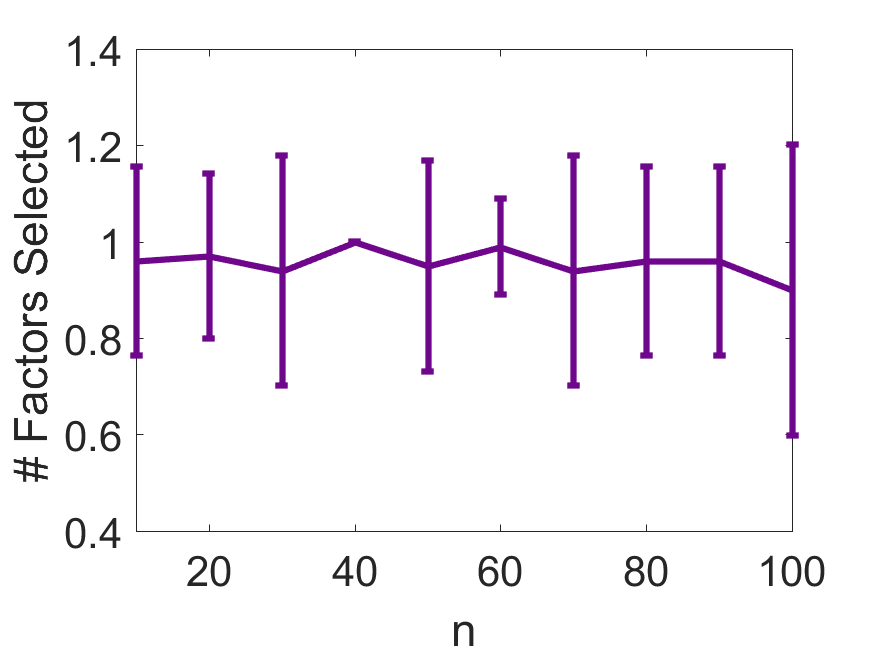}
\end{minipage}%
\begin{minipage}{.33\textwidth}
  \centering
 \includegraphics[scale=0.35]{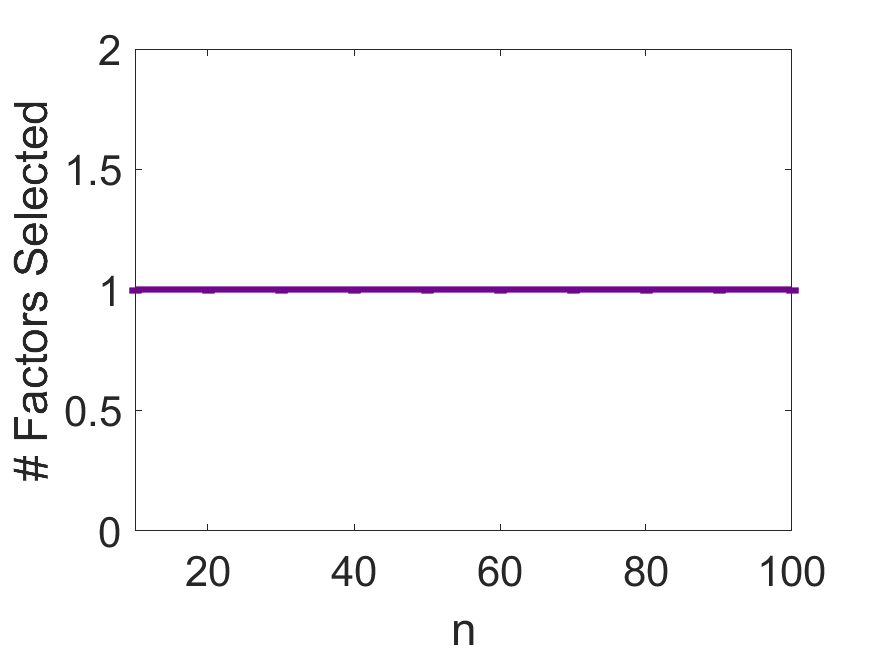}
\end{minipage}
\begin{minipage}{.33\textwidth}
  \centering
 \includegraphics[scale=0.35]{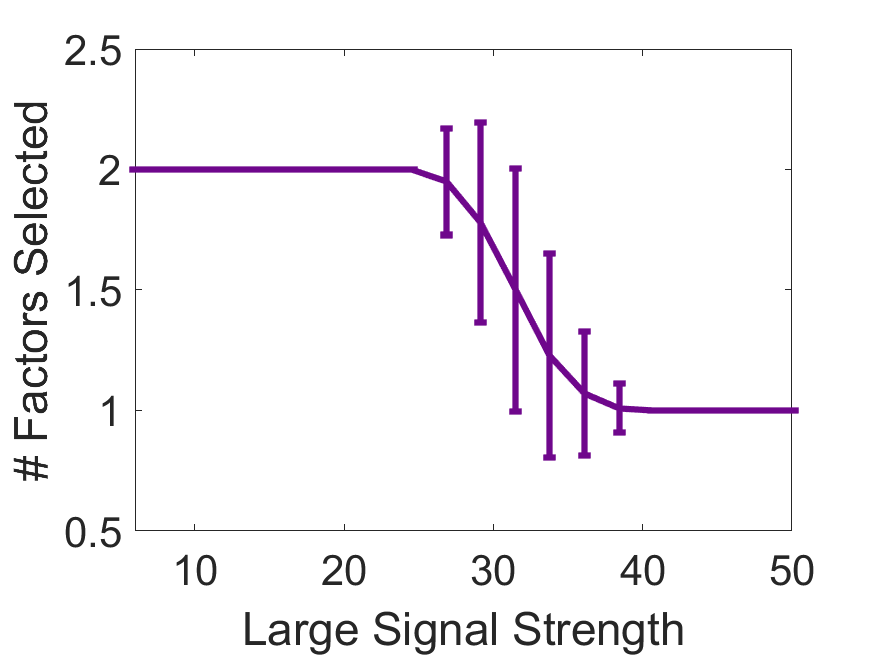}
\end{minipage}
\caption{Mean and SD of number of factors selected by PA as a function of sample size for $p=3$ (left) and $p=1000$ (middle). Same quantity in a 2-factor model as a function of stronger factor value (right).}
\label{fig:PA_dim}
\end{figure}

We provide numerical evidence for our claim that ``PA works when the dimension of the data is large, even when the di.'' Using the same model as in the first simulation, we compare the accuracy of PA for $p=3$ and $p=1000$. We set the signal strength to $\theta=6 \gamma^{1/2}$, which is a perceptible factor. This corresponds to the same signal strength for all $p$. Thus, the two problems are equally hard statistically; or put it differently, the SNR is the same for the two values of $p$.  We vary the sample size from $n=10$ to $n=100$ in steps of 10. 

The results in Fig. \ref{fig:PA_dim} show that PA tends to select the right number of factors  almost without error for $p=1000$, but not for $p=3$. This holds already for $p=10$ (data not shown).  This agrees with our theoretical predictions. Moreover, this also suggests that the requirement on the sample size is not stringent.

\subsection{Effect of strong signals on detectability of weak signals: Shadowing}

We provide numerical evidence for the claim that ``PA selects the relatively important factors.'' Using the same model as in the first simulation, we evaluate the accuracy of PA in a two-factor model. We set the smaller signal strength to $\theta_1=6 \gamma^{1/2}$, which is a perceptible factor. We vary the larger signal strength as $\theta_2=c_2 \gamma^{1/2}$ on a grid between $c_2=6$ and 50.

The results in Fig. \ref{fig:PA_dim} (c) show that PA tends to select the right number of factors  almost without error for $c_2<25$, but it starts making errors above that value. Above $c_2>35$, PA consistently selects only one perceptible factor. Qualitatively, these agree with our theoretical predictions. A strong factor is transformed into noise by PA, thus ``shadowing'' the weaker factor. Quantitatively, in this example the ratio of small-to-large signal strength where PA breaks down is $\theta_1/(\theta_1+\theta_2)  \approx 6/35 \approx 0.17$. According to our theory, this should be on the order of $n^{-1/2}+p^{-1/2} = 0.1$. Thus, our predictions seem quite accurate.

\section{Discussion}
\label{conc}
In this paper we provided a theoretical analysis of parallel analysis (PA). We established precise conditions under which PA consistently selects the perceptible factors for large datasets. We argued that PA works when the dimension of the data is large, and when the factors load on more than just a few variables. 

There are numerous important directions for future research. First, there are variants of PA  developed in applied research \citep[see e.g.,][]{peres2005many, brown2014confirmatory, crawford2010evaluation, gaskin2014exploratory}. When are they useful?  These methods differ in:

\benum
%\item Simulate independent data \citep{horn:1965}. 

\item The test statistic, for instance: singular value gap, fraction of variance explained, robust correlations, loadings \citep{buja:eyob:1992}.

\item The number of permutations, and percentile used: mean of eigenvalues \citep{horn:1965},  other percentiles \citep{buja:eyob:1992,glor:1995}. 

\item Using stepwise testing \citep{horn:1965}. 

\item Using the correlation matrix.
\eenum

Can we understand when they help, and possibly develop improvements? This is especially interesting for tasks \emph{other} than selecting the number of factors, such as estimating the factor loadings. 

Second, what should one do when the noise is  correlated? Independent permutations on the original columns do not generate the correct null distribution. In Sec. \ref{invar} we saw that taking the Fourier transform may help for stationary time series. However, this will need a much more careful analysis. 

Third, an important issue that we have not discussed is the computational cost of PA. The cost of permutations and SVDs for PA can become a problem for ``big data''. Another important issue is the randomness introduced by PA, which can lead to arbitrary decisions. Can we speed up PA, or remove the randomness? \cite{zhou2017eigenvalue} developed such a method, by employing \cite{dobriban2015efficient}'s Spectrode algorithm to approximate the noise level. Can we develop a theoretical understanding of this method, with suitable improvements?

%Fourth, what are the precise conditions under which PA selects the perceptible factors? Is it possible to show that PA works under more general conditions on the signal strength and the delocalization, as discussed in Sec. \ref{optim}?

%Fifth, can we improve the performance of PA when some factors are much stronger than others? In our main results and simulations we saw the somewhat counterintuitive effect that strong factors can lead to overestimation of the noise, so that weaker factors are not selected. This suggests a deflation approach. Perhaps strong factors can be estimated accurately enough to be "removed" from the observation matrix, leading to a more accurate noise estimation? This remains to be seen in future work. 
%
%\section{Acknowledgements}
%
%We thank Andreas Buja, Alexei Onatski, and Art Owen for stimulating discussions. We are very grateful to Jingshu Wang for a careful reading of the manuscript, and numerous helpful suggestions. 

\section{Proofs}
\subsection{Proof of Thm. \ref{fa_cons}}
\label{fa_cons_pf}
We will check that the conditions of the Consistency Lemma \ref{non_asy_pa} hold with probability tending to one. In matrix form, the factor model reads $X = U \Psi^{1/2}\Lambda^\top +Z\Phi^{1/2}$. We first normalize it to have operator norm of unit order: $n^{-1/2}X = n^{-1/2}U \Psi^{1/2}\Lambda^\top +n^{-1/2}Z\Phi^{1/2}$. 

Let us verify the required conditions: 
\benum

\item Signal: Let $\Lambda\Psi^{1/2} = [f_1,\ldots,f_r]$. 

The signal component is 
$S=n^{-1/2} U \Psi^{1/2}\Lambda^\top  = n^{-1/2} \sum_{k=1}^r u_k f_k^\top$

Since there are only a fixed number of factors, it is enough to analyze one term $n^{-1/2} u f^\top = n^{-1/2} |u|_2 |f|_2 \cdot  \tilde u \tilde f^\top$, where $\tilde \cdot$ denotes normalized vectors. Let $\text{e} = n^{-1/2}1$. 

\benum

\item Mean term: The term in the span of $1$ is $ n^{-1/2} |u|_2 |f|_2 \cdot  \tilde u^\top \text{e}\cdot \text{e} \tilde f^\top$.
We need that $n^{-1/2} |u|_2 |f|_2 \cdot |\tilde u^\top \text{e}| \to 0$.

Now, $|f|_2 \le C n^{1/4-\delta/2}$ by assumption. Moreover, if $u_i$ has iid entries with mean 0 and variance 1, then by the LLN $n^{-1/2} |u|_2 \to 1$. Thus it is enough that $n^{1/4-\delta/2}|\tilde u^\top \text{e}|  = |n^{-1/4-\delta/2}(\sum u_i)| \to 0$. This holds by the CLT.

\item Zero-mean term: We need to analyze the term in the orthocomplement of $1$. Let $P = I - \text{e}\text{e}^\top$ be the de-meaning projection operator. Our term is $n^{-1/2} Pu f^\top = n^{-1/2} |f|_2 |Pu|_2 \cdot \widetilde{Pu} \tilde f^\top$

From Thm \ref{sig_cond}, where the low rank part has form $\theta u v^\top$, we need a bound on $\theta (2n^{-1} +|v|_4^4)^{1/4}$. But note that 
$$\theta/[C n^{1/4-\delta/2}] =n^{-1/2} |f|_2/[C n^{1/4-\delta/2}] |Pu|_2 \le n^{-1/2} |u|_2 \to 1$$
 a.s., so we need only 
$$n^{1/4-\delta/2} (2n^{-1} +|v|_4^4)^{1/4} \to 0.$$
 i.e., $n^{1/4-\delta/2}|\tilde f|_4 \to 0$, or also $n^{1/4-\delta/2}|f|_4/|f|_2 \to 0$.

\eenum

\item Noise: We need first that $n^{-1/2}Z\Phi^{1/2}$ has a distribution that is invariant under permutations, as discussed in Sec. \ref{invar}. This holds by inspection.

We need second that  $\|n^{-1/2}Z\Phi^{1/2}\|_2 \to b$. Conditions for this are given in Prop. \ref{noise_op_norm}, and one can verify that the conditions given in the theorem match these. Thus, PA selects all perceptible factors, and no imperceptible factors.

\eenum

\subsection{Proof of Thm. \ref{sig_cond}}
\label{sig_cond_pf}

We need to show that $\|S_\pi\|\to 0$, where $S = n^{-1/2}\theta\cdot 1v^\top+\sum_{i=1}^r \theta_i u_i v_i^\top$. The term $n^{-1/2}\theta\cdot 1v^\top$ is handled by the assumption $\theta\to 0$, so we can focus on the rest, and assume $\theta=0$ from now on. Note that $[\tr (A^\top A)^{k}]^{1/(2k)} = \|A\|_{2k}$ is the Schatten $2k$-norm of $A$. By the triangle inequality for the Schatten norm,  
$\|S_\pi\|_{2k} \le \sum_{i=1}^r \theta_i |(u_i v_i^\top)_\pi|_{2k}.$
Hence, $$\|S_\pi\|_{2k}^{2k} \le \left[\sum_{i=1}^r \theta_i \|(u_i v_i^\top)_\pi\|_{2k}\right]^{2k},$$ therefore
$$[\E\tr (S_\pi^\top S_\pi)^{2k}]^{1/(2k)} \le \left[\E (\sum_{i=1}^r \theta_i D_i^{1/(2k)})^{2k}\right]^{1/(2k)},$$
where $D_i = \E\tr [(u_i v_i^\top)_\pi^\top (u_i v_i^\top)_\pi]^{2k}$. Next, by the triangle inequality for the $\ell_{2k}$ norm $X \to [\E \|X\|^{2k}]^{1/(2k)}$, 
$$[\E (\sum_{i=1}^r \theta_i D_i^{1/(2k)})^{2k}]^{1/(2k)} \le \sum_{i=1}^r \theta_i [\E  D_i]^{1/(2k)}.$$
 Let us focus on bounding one such term $\E D_i$, and denote $u=u_i$, $v=v_i$ for simplicity. Let us write $A=(u v^\top)_\pi$. 

The simplest calculation is the first moment bound $\|A\|^2 \le \tr (A^\top A) $. However, this is not effective, as  $\tr (A^\top A) = \sum_{ij} [\pi(uv^\top)_{ij}]^2 =\sum_{ij} [(uv^\top)_{ij}]^2$ $ = \|uv^\top\|^2 = |u|_2^2|v|_2^2 = 1$, because the permutation of the entries does not change the sum of squares.

\subsubsection{Second moment}

We turn to the second simplest calculation, the second moment bound. This starts with the identity
$$\tr (A^\top A)^2 =\tr A^\top AA^\top A = \sum_{ijkl} A_{ij} A^\top_{jk} A_{kl} A^\top_{li} =\sum_{ijkl} A_{ij} A_{il} A_{kj} A_{kl}.$$
We have $A_{ab} = u_{\pi_b(a)} v_b$. The vectors $u,v$ are fixed, while $\pi_b$ are random and independent across $b$. Thus, if $j \neq l$, then $A_{\cdot j}$ and $A_{\cdot l}$ are independent.  Moreover, the joint distribution of  $(A_{ij},A_{kj})$, for $i \neq k$, is equal to that of $v_j \cdot (u_{\tau_1},u_{\tau_2})$, where $\tau$ is a permutation chosen uniformly at random. Thus, $$(u_{\tau_1},u_{\tau_2}) \sim Unif\{(u_i,u_j): i\neq j \}.$$
With this observation, we can make the following moment calculations:
\benum
\item

$\E A_{ij} = v_j \cdot\E u_{\tau_1} = v_j \cdot\sum_i u_i/n = 0$. 
\item

$\E A_{ij}^2 = v_j^2 \cdot\E u^2_{\tau_1} = v_j^2 \cdot\sum_i u_i^2/n = v_j^2 /n $. 
\item

$\E A_{ij} A_{kj} = v_j^2 \cdot\E u_{\tau_1}u_{\tau_2} = v_j^2 \cdot\sum_{i\neq j} u_i u_j /[n(n-1)] = v_j^2 \cdot [(\sum u_i)^2-1] /[n(n-1)]  = - v_j^2 /[n(n-1)] $. 

\eenum
Therefore, we conclude that
\begin{align*}
 \sum_{ijkl} \E A_{ij} A_{kj} A_{il} A_{kl}
&= \sum_{ik, j\neq l} \E A_{ij} A_{kj} \cdot \E A_{il} A_{kl} + \sum_{ik,j} \E (A_{ij} A_{kj})^2 \\
&= n(n-1) \sum_{j\neq l} \E A_{1j} A_{2j} \cdot \E A_{1l} A_{2l}
+n \sum_{j\neq l} \E A_{1j}^2\cdot \E A_{1l}^2\\
&  +\sum_{j} (\E \sum_{i}A_{ij}^2)^2 \\
 &= n(n-1) \cdot I+n \cdot II+ III.
\end{align*}

Then, we have the following bounds for $I,II$, and $III$: 
\benum
\item
\begin{align*}
I  &=\sum_{j\neq l} \E A_{1j} A_{2j} \cdot \E A_{1l} A_{2l} 
=\sum_{j\neq l} - v_j^2/[n(n-1)]  \cdot ( - v_l^2/[n(n-1)])   \\
&=  1/[n(n-1)]^2 \sum_{j\neq l}v_j^2 v_l^2 = (1-\sum_j v_j^4) /[n(n-1)]^2 \le 1/[n(n-1)]^2
\end{align*}

Note that $I \ge 0$, so $|I| \le 1/[n(n-1)]^2$
\item
$II  = \sum_{j\neq l} \E A_{1j}^2\cdot \E A_{1l}^2 = 1/n^2 \sum_{j\neq l} v_j^2 v_l^2 =  (1-\sum_j v_j^4) /n^2 \le 1/n^2$
\item

$III = \sum_{j} (\E \sum_{i}A_{ij}^2)^2  =  \sum_{j} (n \E A_{1j}^2)^2 = \sum_j v_j^4$

Above we used $\E \sum_{i}A_{ij}^2 = n \E A_{1j}^2  = v_j^2$. 
\eenum

Combining the bounds for $I,II$, and $III$, we get the upper bound
$$
\E\tr (A^\top A)^2 \le 1/[n(n-1)]+1/n+\sum_j v_j^4  = 1/(n-1) +|v|_4^4
$$
In conclusion $\E D_i \le 2/n +|v_i|_4^4$, and $[\E\tr (S_\pi^\top S_\pi)^{4}]^{1/4} \le \sum_{i=1}^r \theta_i [ 1/(n-1)  +|v_i|_4^4]^{1/4}$. This finishes the second moment bound. 

The overall rate at which $\tr(S_\pi^\top S_\pi)^2$ decays is $1/n$. For a.s convergence, we need the bounds to be a summable sequence; thus one can not to prove a.s convergence using only a second moment argument. This motivates us to look at the third moment.

\subsubsection{Third moment}

The third moment bounds proceeds similarly. We start with the identity
$$\tr (A^\top A)^3 =\tr A^\top AA^\top A A^\top A = \sum_{ijklmq} A_{ij} A_{kj} A_{kl} A_{ml} A_{mq} A_{iq}.$$
In this expression, the random variables with this the same second index ($j,l$ or $q$) are dependent, thus there are at most three groups of independent random variables.  There are three cases, depending on how many distinct indices there are among $j,l$ or $q$:

{\bf Three distinct indices: $j\neq l\neq q$}. In this case, we can write the sum over all $j\neq l\neq q$ as
$A_1 = \sum_{ijklmq} \E  A_{ij} A_{kj} \cdot \E A_{kl} A_{ml} \cdot \E A_{mq} A_{iq}.$
We already calculated that $\E  A_{ij} A_{kj} = v_j^2/n \cdot  [\tau+ \delta_{ik} \eta] $, where $\tau = -1/(n-1)$, $\eta = 1 - \tau = n/(n-1)$. Thus, 
$$A_1 = n^{-3} \sum_{j\neq l \neq q} v_j^2v_l^2v_q^2 \cdot  \sum_{ikm} [\tau+ \delta_{ik} \eta]  \cdot  [\tau+ \delta_{km} \eta]  \cdot  [\tau+ \delta_{mi} \eta].$$
Now, we need to evaluate
$$A_2 = \sum_{ikm} [\tau+ \delta_{ik} \eta]  \cdot  [\tau+ \delta_{km} \eta]  \cdot  [\tau+ \delta_{mi} \eta].$$
For the formal calculation, we can factor out $\tau^3$, even though $\tau$ may be 0, and so this may technically not be allowed. However, the formal calculation still leads to the correct answer. If $\tau=0$, the result is $A_2 = n\eta^3$, which agrees with what we get below. Let thus $\zeta = \eta/\tau$, and we want 
\begin{align*}
\tau^{-3}A_2  &= \sum_{ikm} [1+ \delta_{ik} \zeta]  \cdot  [1+ \delta_{km} \zeta]  \cdot  [1+ \delta_{mi} \zeta] \\
&=  \sum_{ik} [1+ \delta_{ik} \zeta]  \cdot  \sum_m [1+ \delta_{km} \zeta]  \cdot  [1+ \delta_{mi} \zeta]\\
&=  \sum_{ik} [1+ \delta_{ik} \zeta]  \cdot  [n+2\zeta +\delta_{ki} \zeta^2] \\
&=  [n+2\zeta] \sum_{ik} [1+ \delta_{ik} \zeta]+   \zeta ^2  \sum_{ik} [1+ \delta_{ik} \zeta] \delta_{ik}\\
&=  [n+2\zeta] \cdot [n^2+n \zeta]+   n \zeta^2   [1+ \zeta]\\
&=  n^3+3n^2\zeta+3n\zeta^2+n \zeta^3.
\end{align*}

Above we used that $\sum_m [1+ \delta_{km} \zeta]  \cdot  [1+ \delta_{mi} \zeta]  = n+2\zeta + \delta_{ki} \zeta^2$. Hence
$$A_2 = n^3 \cdot \tau^3+3n^2 \cdot \tau^2\eta+3n \cdot \tau\eta^2+n \cdot \eta^3 = (n\tau+\eta)^3 +(n-1)\eta^3.$$
However, we also have 
$n\tau+\eta = n\tau +1 -\tau = (n-1)\tau + 1 = 0$, and 
$(n-1)\eta^3 = n^3/(n-1)^2$. Therefore, we conclude that 
$A_2 = n^3/(n-1)^2.$
Going back to the definition of $A_1$, we thus see 
$A_1 = 1/(n-1)^2 \cdot \sum_{j\neq l \neq q} v_j^2v_l^2v_q^2.$

Now $\sum_{j\neq l \neq q} v_j^2v_l^2v_q^2 \le \sum_{jlq} v_j^2v_l^2v_q^2 = |v|_2^6 = 1$, so we conclude that 
$A_1 \le  1/(n-1)^2.$

%\item 
{\bf Two distinct indices: $j= l\neq q$ and the other two symmetric cases.} In this case, we can write the sum as
$B_1 = \sum_{j\neq q, ikm} \E  A_{ij} A_{kj}^2 A_{mj} \cdot \E A_{mq} A_{iq}.$
Now, it is easy to see that the $v_j$ terms contribute a factor of at most $(\sum_j v_j^4)(\sum_j v_j^2) = \sum_j v_j^4 $. In the remainder, it is enough to work with the $u$-part. This equals 
$B_2 = \sum_{ikm} \E  \tau_i \tau_k^2 \tau_m \cdot \E \tau_i \tau_m,$ 
where $\tau$ is a random permutation of the set of values $u_1,\ldots, u_n$. Now we can sum over $k$ first to get 
$$
B_2 = \sum_{im} \E \tau_i \tau_m \cdot \sum _k \E  \tau_i \tau_k^2 \tau_m = \sum_{im} \E \tau_i \tau_m \cdot  \E [ \tau_i \tau_m (\sum _k \tau_k^2)]
$$

However, $\sum _k \tau_k^2 = \sum _k u_k^2=1$ is a deterministic quantity, so we obtain 
$B_2 = \sum_{im} (\E \tau_i \tau_m)^2$.
Now, recall that $\E \tau_i \tau_k = 1/n \cdot  [\tau+ \delta_{ik} \eta] $, where $\tau = -1/(n-1)$, $\eta = 1 - \tau.$
Therefore, 
$$
n^{2} B_2 =  n(n-1) \cdot 1/(n-1)^2+ n \cdot n^2/(n-1)^2.
$$
So $B_2 \le 3/n$ and $B_1 \le 3n^{-1} |v|_4^4$

%\item 
{\bf One unique index: $j= l=q$.} In this case, we can write the sum as
$$C_1 = \sum_{j, ikm} \E  A_{ij}^2 A_{kj}^2 A_{mj}^2  
= \sum_j( \sum_{i} \E  A_{ij}^2)^3
=  \sum_j(v_j^2)^3  = \sum_j v_j^6.$$
Putting together the results from the three cases, we obtain
$$\E\tr (A^\top A)^3 \le 1/(n-1)^2 +9n^{-1} |v|_4^4 + |v|_6^6.$$
This finishes the proof.

\subsubsection{Fourth moment}

The fourth moment bounds proceeds similarly, except the calculation is more complicated. We start with the identity
$$\tr (A^\top A)^4 = \sum_{i_1i_2i_3i_4j_1j_2j_3j_4} A_{i_1j_1} 
A_{i_2j_1} A_{i_2j_2} A_{i_3j_2} A_{i_3j_3} A_{i_4j_3}A_{i_4j_4} A_{i_1j_4}.$$
As before, in this expression, the random variables with this the same second index ($j_{\cdot}$) are dependent, thus there are at most four groups of independent random variables.  There are now four cases, depending on how many distinct indices there are among them.

{\bf Four distinct indices}. In this case, we can write the sum over all $j_s$ as
$$A_1 = \sum_{i_1i_2i_3i_4j_1j_2j_3j_4} \E A_{i_1j_1} 
A_{i_2j_1} \E A_{i_2j_2} A_{i_3j_2} \E A_{i_3j_3} A_{i_4j_3} \E A_{i_4j_4} A_{i_1j_4}.$$
We already calculated that $\E  A_{ij} A_{kj} = v_j^2/n \cdot  [\tau+ \delta_{ik} \eta] $, where $\tau = -1/(n-1)$, $\eta = 1 - \tau = n/(n-1)$. 
Thus, denoting $I = (i_1,i_2,i_3,i_4)$, $J = (j_1,j_2,j_3,j_4)$, $\tilde v_J = v_{j_1}v_{j_2}v_{j_3}v_{j_4}$, $i\in I$ summation over all $i_s$,   and $j\in J$ summation over distinct $j_s$: 
$A_1 = n^{-4} \sum_{J\in S_J}\tilde v_J \cdot  A_2(J),$
where, with $\zeta = \eta/\tau$,
$$\tau^{-3} A_2(J) = \sum_{I \in S_I} \prod_{l=1}^4 [1+ \delta_{i_li_{l+1}}\zeta].$$
This equals
\begin{align*}
&
 \sum_{i_1,i_3} \sum_{i_2} [1+ \delta_{i_1i_2} \zeta] [1+ \delta_{i_3i_2} \zeta]
\cdot
 \sum_{i_4} [1+ \delta_{i_1i_4} \zeta] [1+ \delta_{i_3i_4} \zeta]
\\
&
=  \sum_{i_1,i_3} \left(\sum_{a} [1+ \delta_{i_1a} \zeta] [1+ \delta_{i_3a} \zeta]\right)^2\\
&=  \sum_{ik}  [n+2\zeta +\delta_{ik} \zeta^2]^2 \\
&=  n^4+4n^3\zeta+6n^2\zeta^2+4n\zeta^3+n \zeta^4.
\end{align*}
Here we used identities established in the previous section. This also equals $(n+\zeta)^4+(n-1)\zeta^4$. Hence $A_2 = (n\tau+\eta)^4+(n-1)\eta^3 = n^4/(n-1)^3$. Therefore,
$A_1 = 1/(n-1)^3  \sum_{J\in S_J}\tilde v_J \le  1/(n-1)^3.$

{\bf Three distinct indices: $j_1= j_2$, other $j_s$ different, and the other three symmetric cases.} In this case, we can write the sum as
$$B_1 = \sum_{i_1i_2i_3i_4;j_1j_3j_4} \E A_{i_1j_1} 
A_{i_2j_1}^2 A_{i_3j_1}
\cdot
 \E A_{i_3j_3} A_{i_4j_3} \E A_{i_4j_4} A_{i_1j_4}.$$
As before, it is easy to see that the $v_j$ terms contribute a factor of at most $(\sum_j v_j^4)(\sum_j v_j^2)^2 = \sum_j v_j^4 $. In the remainder, it is enough to work with the $u$-part. This equals 
$$B_2 = \sum_{ikml} \E  \tau_i \tau_k^2 \tau_m \cdot \E \tau_m \tau_l \cdot \E \tau_l \tau_i
=  \sum_{iml} \E  \tau_i \tau_m \cdot \E \tau_m \tau_l \cdot \E \tau_l \tau_i,$$ 
using that $\sum_k\tau_k^2=1$. However, this expression is precisely the one that came up in the calculation of the third moment bound for three distinct indices. There we saw that it equals $1/(n-1)^2$. Hence, 
$B_1 \le |v|_4^4/(n-1)^2,$
and the overall contribution of the terms with three distinct indices is four times this.

{\bf Two distinct indices: $j_1= j_2 \neq j_3=j_4$,  and the other three symmetric cases.} Here we need
$$B_1 = \sum_{i_1i_2i_3i_4;j_1j_3} \E A_{i_1j_1} 
A_{i_2j_1}^2 A_{i_3j_1}
\cdot
 \E A_{i_3j_3} A_{i_4j_3}^2 A_{i_1j_3}.$$
The $v_j$ terms contribute a factor of at most $(\sum_j v_j^4)^2 $. The $u$-part contributes 
$$B_2 = \sum_{ikml} \E  \tau_i \tau_k^2 \tau_m \cdot \E \tau_m  \tau_l^2 \tau_i
=  \sum_{im} (\E  \tau_i \tau_m)^2,$$ 
using that $\sum_k\tau_k^2=1$. In the calculation of the third moment bound for two distinct indices we saw that $B_2 \le 3/n$. Hence, 
$B_1 \le 3|v|_4^8/n.$
The overall contribution is four times this.

{\bf Two distinct indices: $j_1= j_2 =j_3 \neq j_4$,  and the other three symmetric cases.} In this case we need
$$B_1 = \sum_{i_1i_2i_3i_4;j_1j_4}  \E A_{i_1j_1} 
A_{i_2j_1}^2 A_{i_3j_1}^2 A_{i_4j_1}
\cdot \E A_{i_4j_4} A_{i_1j_4}.$$
The $v_j$ terms contribute a factor of at most $(\sum_j v_j^6) (\sum_j v_j^2) = \sum_j v_j^6$. The $u$-part contributes 
$$B_2 = \sum_{ikml} \E  \tau_i \tau_k^2 \tau_m^2  \tau_l \cdot \E \tau_l \tau_i
=  \sum_{im} (\E  \tau_i \tau_m)^2,$$ 
using that $\sum_k\tau_k^2=1$. In the third moment bound for two distinct indices we saw that $B_2 \le 3/n$. Hence, $B_1 \le 3|v|_6^6/n,$
and the overall bound is four times this.

{\bf One distinct index: $j_1= j_2 =j_3 = j_4$,  and the other three symmetric cases.} In this case, we can write the sum as
$$B_1 = \sum_{i_1i_2i_3i_4;j_1j_4}  \E A_{i_1j_1}^2
A_{i_2j_1}^2 A_{i_3j_1}^2 A_{i_4j_1}^2.$$
The $v_j$ contribute a factor of at most $\sum_j v_j^8$, while the $u$-part contributes
$\sum_{ikml} \E  \tau_i^2 \tau_k^2 \tau_m^2  \tau_l^2=  1,$ 
using that $\sum_k\tau_k^2=1$. Hence, 
$B_1 \le |v|_8^8.$

In conclusion we obtain the desired bound
$$\E\tr (A^\top A)^4 \le 1/(n-1)^3 +4/(n-1)^2 |v|_4^4 + 12n^{-1} [|v|_4^8 + |v|_6^6]+ |v|_8^8.$$

\subsection{Proof of Prop. \ref{noise_op_norm}}
\label{noise_op_norm_pf}

The first part essentially follows from \cite[][Cor 6.6]{bai2009spectral}. A small modification is needed to deal with the non-iid-ness, as explained in \cite{dobriban2017optimal}.

For the second part, we will show that $|[\tr T]^{-1/2}  XT^{1/2}|\to 1$. For this, it suffices
 to show that $|[\tr T]^{-1}  XT X^\top-  I_n| \to 0$ in probability (or a.s.). For the convergence in probability, it is in turn enough to show that $\E\tr [X\Sigma X^\top-  I_n]^2 \to 0$, where $\Sigma  = [\tr T]^{-1}  T $. We calculate
$$A = \E \tr [X\Sigma X^\top-  I_n]^2 = \E\tr [X\Sigma X^\top]^2 - 2 \E\tr X\Sigma X^\top +n.$$
Now $X\Sigma X^\top  = \sum_{j=1}^p \sigma_j x_j x_j^\top $, where the $x_j$ are independent $n\times 1$ random vectors whose entries are iid random variables (whose distribution may depend on $j$). They collect the $j$-th coordinates of the observed data. So, $\E \tr X\Sigma X^\top  = \sum_{j=1}^p \sigma_j  \E | x_j|^2  = n  \sum_{j=1}^p \sigma_j = n.$
Also, $$\E\tr [X\Sigma X^\top]^2 
= \E \tr [\sum_{j=1}^p \sigma_j x_j x_j^\top][\sum_{k=1}^p \sigma_k x_k x_k^\top] 
=\sum_{j,k=1}^p \sigma_j\sigma_k  \E [x_j^\top x_k]^2.$$
To evaluate this expression, we need to find $\E [x_j^\top x_k]^2$. If $j \neq k$, then $x_j$ and $x_k$ are independent, and we can take expectation over $j$ first, to get 
$\E [x_j^\top x_k]^2 =\E \tr [x_j x_j^\top x_k x_k^\top] = \E \tr [x_k x_k^\top] = n.$
This leads to 
$$\E\tr [X\Sigma X^\top]^2 
=n \sum_{j,k=1}^p \sigma_j\sigma_k + \sum_{j=1}^p \sigma_j^2 [\E |x_j|^4 - n].$$
Therefore we find that $A  = \sum_{j=1}^p \sigma_j^2 [\E |x_j|^4 - n]$. Thus, we need to show that $A\to 0$. Since $\sigma_j \le C p^{-1}$ for all $j$, 
$$A = p^{-2} \sum_{i=1}^n \sum_{j=1}^p (\E x_{ij}^4 - 1) \le p^{-2} \cdot Cnp = C n/p \to 0$$
 if $n/p \to 0$, since the 4-th moments are bounded.  This shows that $n/p\to0$ guarantees convergence in probability, and finishes the proof of (2A). If in addition $n/p \le 1/n^{1+\ep}$, then by the Borel-Cantelli lemma we conclude that $|[\tr T]^{-1}  XT X^\top-  I_n| \to 0$ a.s., as needed. This  finishes the proof of (2B). Therefore, the proof of the proposition is complete. 
 
%\begin{supplement}
%\sname{Supplement A}\label{suppA}
%\stitle{Title of the Supplement A}
%\slink[url]{http://www.e-publications.org/ims/support/dowload/imsart-ims.zip}
%\sdescription{Z}
%\end{supplement}

{\small
\setlength{\bibsep}{0.2pt plus 0.3ex}
\bibliographystyle{plainnat-abbrev}
\bibliography{references}

\begin{thebibliography}{57}
\providecommand{\natexlab}[1]{#1}
\providecommand{\url}[1]{\texttt{#1}}
\expandafter\ifx\csname urlstyle\endcsname\relax
  \providecommand{\doi}[1]{doi: #1}\else
  \providecommand{\doi}{doi: \begingroup \urlstyle{rm}\Url}\fi

\bibitem[Anderson(2003)]{anderson1958introduction}
T.~W. Anderson.
\newblock \emph{An Introduction to Multivariate Statistical Analysis}.
\newblock Wiley New York, 2003.

\bibitem[Bai and Li(2012)]{bai2012statistical}
J.~Bai and K.~Li.
\newblock Statistical analysis of factor models of high dimension.
\newblock \emph{The Annals of Statistics}, 40\penalty0 (1):\penalty0 436--465,
  2012.

\bibitem[Bai and Ng(2008)]{bai2008large}
J.~Bai and S.~Ng.
\newblock \emph{Large dimensional factor analysis}.
\newblock Now Publishers Inc, 2008.

\bibitem[Bai and Ding(2012)]{bai2012estimation}
Z.~Bai and X.~Ding.
\newblock Estimation of spiked eigenvalues in spiked models.
\newblock \emph{Random Matrices: Theory and Applications}, 1\penalty0
  (02):\penalty0 1150011, 2012.

\bibitem[Bai and Silverstein(2009)]{bai2009spectral}
Z.~Bai and J.~W. Silverstein.
\newblock \emph{Spectral analysis of large dimensional random matrices}.
\newblock Springer Series in Statistics. Springer, 2009.

\bibitem[Baik et~al.(2005)Baik, Ben~Arous, and P{\'e}ch{\'e}]{baik2005phase}
J.~Baik, G.~Ben~Arous, and S.~P{\'e}ch{\'e}.
\newblock Phase transition of the largest eigenvalue for nonnull complex sample
  covariance matrices.
\newblock \emph{Annals of Probability}, 33\penalty0 (5):\penalty0 1643--1697,
  2005.

\bibitem[Bartlett(1950)]{bartlett1950tests}
M.~S. Bartlett.
\newblock Tests of significance in factor analysis.
\newblock \emph{British Journal of Mathematical and Statistical Psychology},
  3\penalty0 (2):\penalty0 77--85, 1950.

\bibitem[Benaych-Georges and Nadakuditi(2012)]{benaych2012singular}
F.~Benaych-Georges and R.~R. Nadakuditi.
\newblock The singular values and vectors of low rank perturbations of large
  rectangular random matrices.
\newblock \emph{Journal of Multivariate Analysis}, 111:\penalty0 120--135,
  2012.

\bibitem[Brockwell and Davis(2009)]{brockwell2009time}
P.~J. Brockwell and R.~A. Davis.
\newblock \emph{Time series: theory and methods}.
\newblock Springer, 2009.

\bibitem[Brown(2014)]{brown2014confirmatory}
T.~A. Brown.
\newblock \emph{Confirmatory factor analysis for applied research}.
\newblock Guilford Publications, 2014.

\bibitem[Buja and Eyuboglu(1992)]{buja:eyob:1992}
A.~Buja and N.~Eyuboglu.
\newblock Remarks on parallel analysis.
\newblock \emph{Multivariate behavioral research}, 27\penalty0 (4):\penalty0
  509--540, 1992.

\bibitem[Cattell(1966)]{cattell1966scree}
R.~B. Cattell.
\newblock The scree test for the number of factors.
\newblock \emph{Multivariate behavioral research}, 1\penalty0 (2):\penalty0
  245--276, 1966.

\bibitem[Churchill~Jr(1979)]{churchill1979paradigm}
G.~A. Churchill~Jr.
\newblock A paradigm for developing better measures of marketing constructs.
\newblock \emph{Journal of marketing research}, pages 64--73, 1979.

\bibitem[Costello and Osborne(2005)]{costello2005best}
A.~B. Costello and J.~W. Osborne.
\newblock Best practices in exploratory factor analysis: Four recommendations
  for getting the most from your analysis.
\newblock \emph{Practical assessment, research \& evaluation}, 10\penalty0
  (7):\penalty0 1--9, 2005.

\bibitem[Crawford et~al.(2010)Crawford, Green, Levy, Lo, Scott, Svetina, and
  Thompson]{crawford2010evaluation}
A.~V. Crawford, S.~B. Green, R.~Levy, W.-J. Lo, L.~Scott, D.~Svetina, and M.~S.
  Thompson.
\newblock Evaluation of parallel analysis methods for determining the number of
  factors.
\newblock \emph{Educational and Psychological Measurement}, 70\penalty0
  (6):\penalty0 885--901, 2010.

\bibitem[Dobriban(2015)]{dobriban2015efficient}
E.~Dobriban.
\newblock Efficient computation of limit spectra of sample covariance matrices.
\newblock \emph{Random Matrices: Theory and Applications}, 04\penalty0
  (04):\penalty0 1550019, 2015.

\bibitem[Dobriban and Owen(2017)]{dobriban2017deterministic}
E.~Dobriban and A.~B. Owen.
\newblock Deterministic parallel analysis: An improved method for selecting the
  number of factors and principal components.
\newblock \emph{arXiv preprint arXiv:1711.04155. To appear in JRSS-B}, 2017.

\bibitem[Dobriban and Wager(2015)]{dobriban2015high}
E.~Dobriban and S.~Wager.
\newblock High-dimensional asymptotics of prediction: Ridge regression and
  classification.
\newblock \emph{arXiv preprint arXiv:1507.03003, to appear in the Annals of
  Statistics}, 2015.

\bibitem[Dobriban et~al.(2017)Dobriban, Leeb, and Singer]{dobriban2017optimal}
E.~Dobriban, W.~Leeb, and A.~Singer.
\newblock Optimal prediction in the linearly transformed spiked model.
\newblock \emph{arXiv preprint arXiv:1709.03393}, 2017.

\bibitem[Fabrigar et~al.(1999)Fabrigar, Wegener, MacCallum, and
  Strahan]{fabrigar1999evaluating}
L.~R. Fabrigar, D.~T. Wegener, R.~C. MacCallum, and E.~J. Strahan.
\newblock Evaluating the use of exploratory factor analysis in psychological
  research.
\newblock \emph{Psychological methods}, 4\penalty0 (3):\penalty0 272, 1999.

\bibitem[Fan et~al.(2011)Fan, Liao, and Mincheva]{fan2011highd}
J.~Fan, Y.~Liao, and M.~Mincheva.
\newblock High dimensional covariance matrix estimation in approximate factor
  models.
\newblock \emph{Annals of Statistics}, 39\penalty0 (6):\penalty0 3320, 2011.

\bibitem[Friedman et~al.(2009)Friedman, Hastie, and
  Tibshirani]{friedman2009elements}
J.~Friedman, T.~Hastie, and R.~Tibshirani.
\newblock \emph{The elements of statistical learning}.
\newblock Springer series in statistics, 2009.

\bibitem[Gaskin and Happell(2014)]{gaskin2014exploratory}
C.~J. Gaskin and B.~Happell.
\newblock On exploratory factor analysis: A review of recent evidence, an
  assessment of current practice, and recommendations for future use.
\newblock \emph{International journal of nursing studies}, 51\penalty0
  (3):\penalty0 511--521, 2014.

\bibitem[Gerard and Stephens(2017)]{gerard2017unifying}
D.~Gerard and M.~Stephens.
\newblock Unifying and generalizing methods for removing unwanted variation
  based on negative controls.
\newblock \emph{arXiv preprint arXiv:1705.08393}, 2017.

\bibitem[Glorfeld(1995)]{glor:1995}
L.~W. Glorfeld.
\newblock An improvement on {Horn's} parallel analysis methodology for
  selecting the correct number of factors to retain.
\newblock \emph{Educational and psychological measurement}, 55\penalty0
  (3):\penalty0 377--393, 1995.

\bibitem[Goetz et~al.(2008)Goetz, Tilley, Shaftman, Stebbins, Fahn,
  Martinez-Martin, Poewe, Sampaio, Stern, Dodel, et~al.]{goetz2008movement}
C.~G. Goetz, B.~C. Tilley, S.~R. Shaftman, G.~T. Stebbins, S.~Fahn,
  P.~Martinez-Martin, W.~Poewe, C.~Sampaio, M.~B. Stern, R.~Dodel, et~al.
\newblock Movement disorder society-sponsored revision of the unified
  parkinson's disease rating scale (mds-updrs): Scale presentation and
  clinimetric testing results.
\newblock \emph{Movement disorders}, 23\penalty0 (15):\penalty0 2129--2170,
  2008.

\bibitem[Green et~al.(2012)Green, Levy, Thompson, Lu, and
  Lo]{green2012proposed}
S.~B. Green, R.~Levy, M.~S. Thompson, M.~Lu, and W.-J. Lo.
\newblock A proposed solution to the problem with using completely random data
  to assess the number of factors with parallel analysis.
\newblock \emph{Educational and Psychological Measurement}, 72\penalty0
  (3):\penalty0 357--374, 2012.

\bibitem[Hayton et~al.(2004)Hayton, Allen, and Scarpello]{hayton2004factor}
J.~C. Hayton, D.~G. Allen, and V.~Scarpello.
\newblock Factor retention decisions in exploratory factor analysis: A tutorial
  on parallel analysis.
\newblock \emph{Organizational research methods}, 7\penalty0 (2):\penalty0
  191--205, 2004.

\bibitem[Horn(1965)]{horn:1965}
J.~L. Horn.
\newblock A rationale and test for the number of factors in factor analysis.
\newblock \emph{Psychometrika}, 30\penalty0 (2):\penalty0 179--185, 1965.

\bibitem[Johnstone(2001)]{johnstone2001distribution}
I.~M. Johnstone.
\newblock On the distribution of the largest eigenvalue in principal components
  analysis.
\newblock \emph{Annals of Statistics}, 29\penalty0 (2):\penalty0 295--327,
  2001.

\bibitem[Jolliffe(2002)]{jolliffe2002principal}
I.~Jolliffe.
\newblock \emph{Principal component analysis}.
\newblock Wiley Online Library, 2002.

\bibitem[Josse and Husson(2012)]{josse2012selecting}
J.~Josse and F.~Husson.
\newblock Selecting the number of components in principal component analysis
  using cross-validation approximations.
\newblock \emph{Computational Statistics \& Data Analysis}, 56\penalty0
  (6):\penalty0 1869--1879, 2012.

\bibitem[Kaiser(1960)]{kaiser1960application}
H.~F. Kaiser.
\newblock The application of electronic computers to factor analysis.
\newblock \emph{Educational and psychological measurement}, 20\penalty0
  (1):\penalty0 141--151, 1960.

\bibitem[Kritchman and Nadler(2008)]{kritchman2008determining}
S.~Kritchman and B.~Nadler.
\newblock Determining the number of components in a factor model from limited
  noisy data.
\newblock \emph{Chemometrics and Intelligent Laboratory Systems}, 94\penalty0
  (1):\penalty0 19--32, 2008.

\bibitem[Leek and Storey(2007)]{leek2007capturing}
J.~T. Leek and J.~D. Storey.
\newblock Capturing heterogeneity in gene expression studies by surrogate
  variable analysis.
\newblock \emph{PLoS genetics}, 3\penalty0 (9):\penalty0 e161, 2007.

\bibitem[Leek and Storey(2008)]{leek2008general}
J.~T. Leek and J.~D. Storey.
\newblock A general framework for multiple testing dependence.
\newblock \emph{Proceedings of the National Academy of Sciences}, 105\penalty0
  (48):\penalty0 18718--18723, 2008.

\bibitem[Lin et~al.(2016)Lin, Yang, Zhu, Duchi, Fu, Wang, Jiang, Zamanighomi,
  Xu, Li, et~al.]{lin2016simultaneous}
Z.~Lin, C.~Yang, Y.~Zhu, J.~Duchi, Y.~Fu, Y.~Wang, B.~Jiang, M.~Zamanighomi,
  X.~Xu, M.~Li, et~al.
\newblock Simultaneous dimension reduction and adjustment for confounding
  variation.
\newblock \emph{Proceedings of the National Academy of Sciences}, 113\penalty0
  (51):\penalty0 14662--14667, 2016.

\bibitem[Nadakuditi(2014)]{nadakuditi2014optshrink}
R.~R. Nadakuditi.
\newblock Optshrink: An algorithm for improved low-rank signal matrix denoising
  by optimal, data-driven singular value shrinkage.
\newblock \emph{IEEE Transactions on Information Theory}, 60\penalty0
  (5):\penalty0 3002--3018, 2014.

\bibitem[Nadler(2008)]{nadler2008finite}
B.~Nadler.
\newblock Finite sample approximation results for principal component analysis:
  A matrix perturbation approach.
\newblock \emph{The Annals of Statistics}, 36\penalty0 (6):\penalty0
  2791--2817, 2008.

\bibitem[Onatski(2009)]{onatski2009testing}
A.~Onatski.
\newblock Testing hypotheses about the number of factors in large factor
  models.
\newblock \emph{Econometrica}, 77\penalty0 (5):\penalty0 1447--1479, 2009.

\bibitem[Onatski(2010)]{onatski2010determining}
A.~Onatski.
\newblock Determining the number of factors from empirical distribution of
  eigenvalues.
\newblock \emph{The Review of Economics and Statistics}, 92\penalty0
  (4):\penalty0 1004--1016, 2010.

\bibitem[Onatski(2012)]{onatski2012asymptotics}
A.~Onatski.
\newblock Asymptotics of the principal components estimator of large factor
  models with weakly influential factors.
\newblock \emph{Journal of Econometrics}, 168\penalty0 (2):\penalty0 244--258,
  2012.

\bibitem[Onatski et~al.(2013)Onatski, Moreira, and
  Hallin]{onatski2013asymptotic}
A.~Onatski, M.~J. Moreira, and M.~Hallin.
\newblock Asymptotic power of sphericity tests for high-dimensional data.
\newblock \emph{The Annals of Statistics}, 41\penalty0 (3):\penalty0
  1204--1231, 2013.

\bibitem[Owen and Wang(2016)]{owen:wang:2016}
A.~B. Owen and J.~Wang.
\newblock Bi-cross-validation for factor analysis.
\newblock \emph{Statistical Science}, 31\penalty0 (1):\penalty0 119--139, 2016.

\bibitem[Parasuraman et~al.(1988)Parasuraman, Zeithaml, and
  Berry]{parasuraman1988servqual}
A.~Parasuraman, V.~A. Zeithaml, and L.~L. Berry.
\newblock Servqual: A multiple-item scale for measuring consumer perceptions of
  service quality.
\newblock \emph{Journal of retailing}, 64\penalty0 (1):\penalty0 12, 1988.

\bibitem[Paul(2007)]{paul2007asymptotics}
D.~Paul.
\newblock Asymptotics of sample eigenstructure for a large dimensional spiked
  covariance model.
\newblock \emph{Statistica Sinica}, 17\penalty0 (4):\penalty0 1617--1642, 2007.

\bibitem[Paul and Aue(2014)]{paul2014random}
D.~Paul and A.~Aue.
\newblock Random matrix theory in statistics: A review.
\newblock \emph{Journal of Statistical Planning and Inference}, 150:\penalty0
  1--29, 2014.

\bibitem[Peres-Neto et~al.(2005)Peres-Neto, Jackson, and Somers]{peres2005many}
P.~R. Peres-Neto, D.~A. Jackson, and K.~M. Somers.
\newblock How many principal components? stopping rules for determining the
  number of non-trivial axes revisited.
\newblock \emph{Computational Statistics \& Data Analysis}, 49\penalty0
  (4):\penalty0 974--997, 2005.

\bibitem[Raiche et~al.(2010)Raiche, Magis, and Raiche]{raiche2010package}
G.~Raiche, D.~Magis, and M.~G. Raiche.
\newblock Package ‘nfactors’.
\newblock 2010.

\bibitem[Spearman(1904)]{spearman1904general}
C.~Spearman.
\newblock "{General Intelligence}", objectively determined and measured.
\newblock \emph{The American Journal of Psychology}, 15\penalty0 (2):\penalty0
  201--292, 1904.

\bibitem[Stewart(1981)]{stewart1981application}
D.~W. Stewart.
\newblock The application and misapplication of factor analysis in marketing
  research.
\newblock \emph{Journal of Marketing Research}, pages 51--62, 1981.

\bibitem[Thurstone(1947)]{thurstone1947multiple}
L.~L. Thurstone.
\newblock \emph{Multiple-factor analysis}.
\newblock University of Chicago Press, 1947.

\bibitem[Velicer(1976)]{velicer1976determining}
W.~F. Velicer.
\newblock Determining the number of components from the matrix of partial
  correlations.
\newblock \emph{Psychometrika}, 41\penalty0 (3):\penalty0 321--327, 1976.

\bibitem[Vershynin(2010)]{vershynin2010introduction}
R.~Vershynin.
\newblock Introduction to the non-asymptotic analysis of random matrices.
\newblock \emph{arXiv preprint arXiv:1011.3027}, 2010.

\bibitem[Yao et~al.(2015)Yao, Bai, and Zheng]{yao2015large}
J.~Yao, Z.~Bai, and S.~Zheng.
\newblock \emph{Large Sample Covariance Matrices and High-Dimensional Data
  Analysis}.
\newblock Cambridge University Press, 2015.

\bibitem[Zhou et~al.(2017)Zhou, Marron, and Wright]{zhou2017eigenvalue}
Y.-H. Zhou, J.~Marron, and F.~A. Wright.
\newblock Eigenvalue significance testing for genetic association.
\newblock \emph{Biometrics}, 2017.

\bibitem[Zwick and Velicer(1986)]{zwic:veli:1986}
W.~R. Zwick and W.~F. Velicer.
\newblock Comparison of five rules for determining the number of components to
  retain.
\newblock \emph{Psychological bulletin}, 99\penalty0 (3):\penalty0 432, 1986.

\end{thebibliography}
}

\end{document}